\documentclass[11pt]{amsart}

\usepackage[english]{babel}

\usepackage{amsmath,amsthm, amssymb, bbm,  color}
\usepackage{graphicx}
\usepackage{fullpage}

\newcommand{\R}{{{\mathbb {R}}}}

\newcommand{\Z}{{\mathbb Z}}

\newcommand{\U}{{\mathbb U}}
\newcommand{\HH}{{\mathbb H}}

\newcommand{\epsil}{u}
\newcommand{\cech}{\check}

\newtheorem{Theorem}{Theorem}

\newtheorem {lemma} [Theorem]    {Lemma}
\newtheorem {corollary}  [Theorem]    {Corollary}

\newtheorem {proposition}[Theorem]    {Proposition}
\newtheorem {theorem}[Theorem]    {Theorem}

\let\oldmarginpar\marginpar
\renewcommand\marginpar[1]{\-\oldmarginpar[\raggedleft\scriptsize #1]%
{\raggedright\scriptsize #1}}

\newcommand {\eps}{\varepsilon}

\begin{document}

\title{Coupling the Gaussian Free fields with free and with zero boundary conditions via common 
 level lines}
\author{Wei Qian}
\author{Wendelin Werner}
\address {
 Center of Mathematical Sciences, 
Wilberforce Rd., Cambridge CB3 0WB, United Kingdom}
\email {wq214@cam.ac.uk}

\address {
Department of Mathematics,
ETH Z\"urich, R\"amistr. 101,
8092 Z\"urich, Switzerland}

\email{wendelin.werner@math.ethz.ch}

\begin {abstract}
We point out a new simple way to couple the Gaussian Free Field (GFF) with free boundary conditions in a two-dimensional domain with the GFF with zero boundary conditions in the same domain: 
Starting from the latter, one just has to sample at random all the signs of the height gaps on its boundary-touching zero-level lines (these signs are alternating for the zero-boundary GFF) in order to obtain a free boundary GFF. 
Constructions and couplings of the free boundary GFF and its level lines via soups of reflected Brownian loops and their clusters
are also discussed. Such considerations show for instance that in a domain with an axis of symmetry, if one looks at the overlay of a single usual Conformal Loop Ensemble CLE$_3$ with its own symmetric image, one obtains the CLE$_4$-type collection of level lines of a GFF with mixed zero/free boundary conditions in the half-domain. 
\end {abstract}

\maketitle

\tableofcontents

\section{Introduction}

A large number of recent  papers (see for instance \cite {SchSh,SchSh2,Dub,SheffieldQZ,MT,MS1,MS2,MS3} and references therein)  have highlighted the close connection between Schramm-Loewner evolutions (SLE) and their variants such as SLE$_\kappa (\rho)$ processes or the conformal loop ensembles (CLE) with the 
two-dimensional Gaussian free field (GFF). This has led to a much better understanding of the geometric structures underlying the GFF, as well as to new results about SLE and CLE.

A special role is played by the SLE$_4$ process, the SLE$_4(\rho)$ processes and the conformal loop ensemble CLE$_4$, because, as pointed out by Schramm and Sheffield 
for SLE$_4 (\rho)$ curves in \cite {SchSh,SchSh2} (see also Dub\'edat \cite {Dub}),
they can be viewed as level lines of the GFF with constant or piecewise constant boundary conditions,
or more precisely as lines along which the GFF has a certain height-gap $2 \lambda = \sqrt {\pi /2}$.  

Gaussian Free Fields with other natural boundary conditions than constant boundary conditions are of course of wide interest. 
One prime example is the GFF with free boundary conditions in a domain $D$ (that we will refer to from now on as 
the GFF with Neumann boundary conditions, or the Neumann GFF -- we will also refer to the GFF with zero boundary conditions as the Dirichlet GFF). 
Recall that, as opposed to the Dirichlet GFF, the Neumann GFF is only defined up to an additive constant (one way to think of this is that one knows the gradient 
of a generalized fuction, but not the function itself). However, away from the boundary of $D$, the (generalized) gradient of this Neumann GFF is absolutely continuous with respect to 
the (generalized) gradient of the Dirichlet GFF, so that it is also possible to make sense of the level-lines of this field, that are also SLE$_4$-type curves
(the precise description for mixed Dirichlet-Neumann boundary conditions has for instance been given by Izyurov and Kyt\"ol\"a in \cite {IK}). 
The Neumann GFF also plays a central role in the  zipping/welding approach to SLE (see Sheffield's quantum zipper \cite {SheffieldQZ}), or via  the 
interplay between SLE paths and Liouville Quantum Gravity, see for instance \cite {MT}.

It seems that some simple CLE-type constructions or descriptions of the Neumann GFF and their consequences 
may have been overlooked, and one of the goals of the present paper is to fill this gap. More precisely, in this direction we will: 

\medbreak
{\em (1) Describe the collection of level lines of the Neumann GFF.} The fact that this can be done is not really surprising (for instance given the results in \cite {IK}), but the consequence (2) was maybe more unexpected. 

\medbreak
{\em (2) Point out a simple coupling of the Neumann GFF with a Dirichlet-GFF in the same domain}: The difference between the two fields in this coupling is a function that is constant by parts, and  the two fields will share a number of level lines. Only the signs of the height gaps for the boundary touching level lines of height $0$ for the Dirichlet GFF will differ, so that the difference between the two GFFs will take its values in $4\lambda \Z$ (see Figure \ref {pic00} for a sketch). More precisely, start with a Dirichlet GFF $\Gamma$ and define the collection $A$ of its boundary touching level lines at height $0$ (these are $\lambda$ versus $-\lambda$ height-gap lines). The complement of $A$ consists of a family of cells in which $\Gamma$ behaves like 
a GFF with constant boundary conditions $\lambda$ or $-\lambda$ as on the left of Figure \ref {pic00} (this means that inside each cell, one considers a GFF with Dirichlet 
boundary conditions and adds the constant function $\lambda$ or $-\lambda$ depending on the cell). Here, the height gaps between two neighboring cells are alternating 
between $2 \lambda$ and $- 2 \lambda$ so that the height of all cells stays in  $\{- \lambda, + \lambda \}$.
Now, toss an independent fair coin for each arc of $A$ in order to decide the corresponding height-gap between the two cells that it separates. This leads to the picture on the right of Figure \ref {pic00}. Theorem \ref {mainthm} will state that adding an independent GFF in each cell to these new heights defines a scalar field which is exactly  a realization of a Neumann GFF
(i.e., its gradient is that of a Neumann GFF). In other words, if one adds to the Dirichlet GFF the function that is constant in each cell and equal to the difference between the new and old heights in 
that cell, one does obtain a Neumann GFF.

\begin{figure}[ht!]
\includegraphics[scale=0.7]{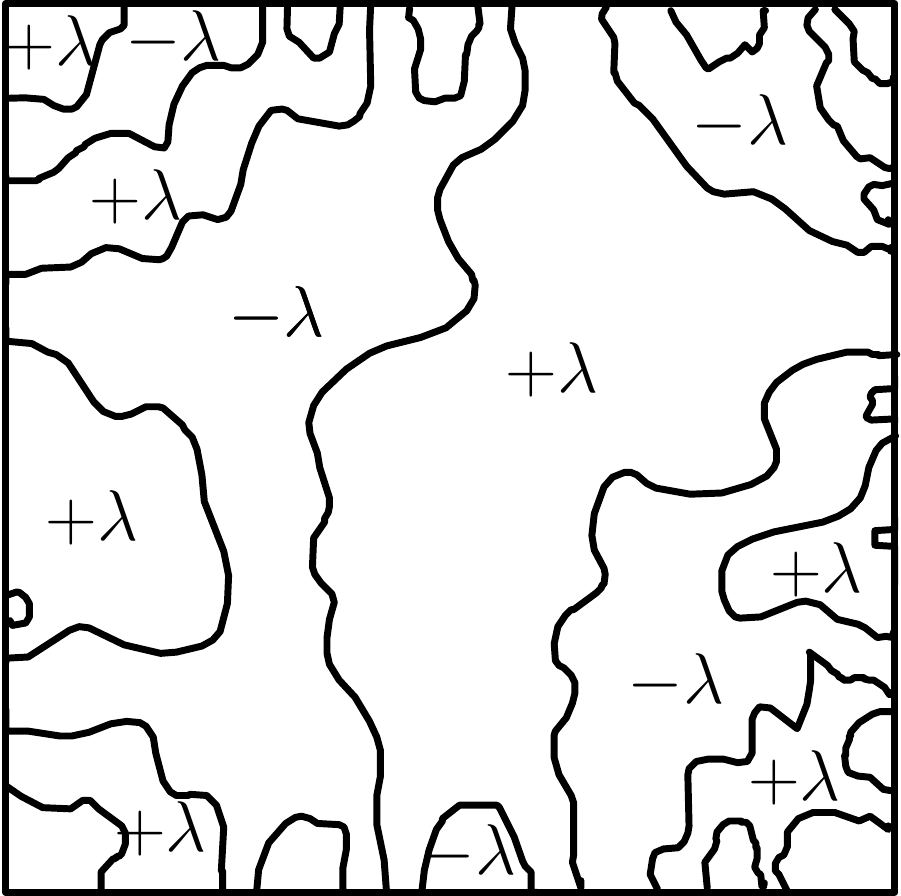}
\quad
\includegraphics[scale=0.7]{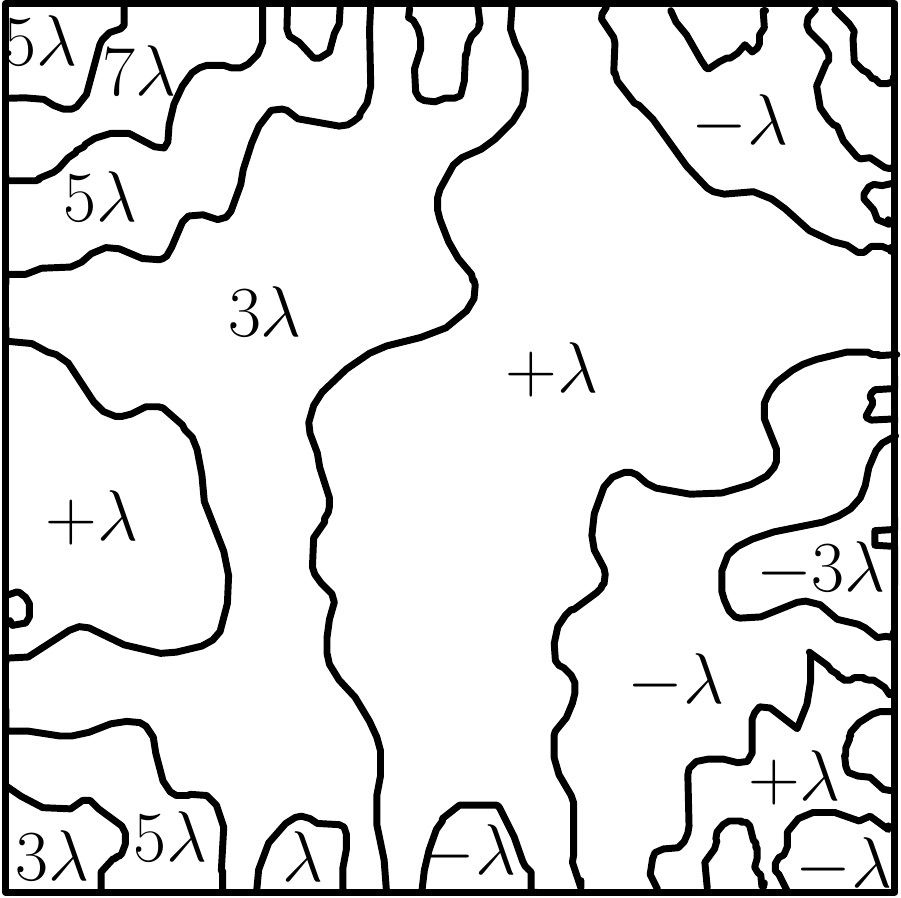}
  \caption {Sketch of the coupling: On the left, the boundary touching level lines of the GFF with Dirichlet boundary conditions. Adding to these values an independent GFF in each of the cells constructs a GFF with Dirichlet boundary conditions in the square. 
  On the right, we keep the same set of boundary touching level lines but the signs of the height gaps are chosen at random. Adding to these values an independent GFF in each of the cells constructs a GFF with Neumann boundary conditions in the square}
  \label{pic00}\end{figure}

\medbreak
The existence and properties of this 
coupling between Neumann GFF and Dirichlet GFF  shows again -- if needed -- how natural these SLE$_4$ level lines are in order to 
connect and understand these two fields. 
%Note that this type of coupling results seems a priori specific to the two-dimensional case. 

\medbreak

Another approach to the coupling between SLE$_4$ (or CLE$_4$) with the GFF with Dirichlet boundary conditions uses the Brownian loop-soup introduced in \cite {LW}: 
On the one hand CLE$_4$ can be constructed as outermost boundaries of Brownian loop-soup clusters of appropriate intensity \cite {ShW}, so that it is a deterministic function of this loop-soup. 
On the other hand \cite {LJ}, one can show that the (appropriately renormalized) occupation time measure of the Brownian loop soup is distributed like the (appropriately defined) square of a Dirichlet GFF -- 
so that this square of the Dirichlet GFF is also a deterministic function of the loop-soup. Furthermore \cite {Lupu}, it is possible to use the loop-soup in order to reconstruct the Dirichlet GFF itself (loosely speaking, 
by tossing one independent coin for each loop-soup cluster to choose a sign). Combining these constructions provides a coupling of CLE$_4$ with the GFF. 
As explained in our paper \cite {QianW},  it is possible to show (using also some further results of Lupu \cite {Lupu2} and the relation to Dynkin's isomorphism theorem) that this coupling can be made to coincide with the Miller-Sheffield coupling where the CLE$_4$ are the level lines at height $\pm \lambda$ of the Dirichlet GFF (see \cite {ASW} for this). We will also address in the present paper the free boundary GFF counterparts of all these facts. Not surprisingly, it will involve soups of reflected Brownian loops: 

\medbreak

{\em (3) We will explain that all the features about loop-soup clusters, their decompositions and the relation to the GFF, as shown in \cite {QianW} for Dirichlet boundary conditions have natural analogs for  Neumann boundary conditions.}

\medbreak
Recall that one consequence of the Brownian loop-soup approach to CLE$_\kappa$ is that it enables to derive relations between the various CLE$_\kappa$'s that seem out of reach by an SLE$_\kappa$-based definition of the conformal loop ensembles. More precisely, since a CLE$_{\kappa(c)}$ for some explicit function $\kappa (c) \in (8/3, 4]$ is the collection of outermost boundaries of loop-soup clusters with intensity $c \le 1$, one can construct it from two independent samples of CLE$_{\kappa (c')}$ and CLE$_{\kappa (c'')}$ when $c' + c'' = c$ by looking at the clusters in the union of these two CLEs (a noteworthy example is the fact that one can reconstruct a CLE$_4$ using the overlay of two independent CLE$_3$'s, which are known to be the scaling limit of Ising model interfaces, see \cite {BH} and the references therein).  
This makes it natural to define also such  ``semi-groups'' of CLEs for Neumann-type boundary conditions: 

\medbreak
{\em 
(4) We will explain how  to naturally define Neumann-type CLE$_\kappa$'s for all $\kappa \in (8/3, 4]$. This for instance leads to simple coupling between a usual (Dirichlet) CLE$_3$
in the unit disc and a GFF with mixed Dirichlet-Neumann conditions in the half-disc (Dirichlet on the half-circle, Neumann on the diameter $I$) 
by looking at the overlay in the half-disc of the CLE$_3$ with its symmetric image with respect to the diameter $I$.}

\medbreak

The organization of the paper is the following: 
We will first recall background on 
the collections of level lines for the Dirichlet GFF in Section \ref {S2}. In Section \ref {S3}, we state the coupling of the Neumann GFF with the Dirichlet GFF, Theorem \ref {mainthm}, and make various 
comments. We then prove  Theorem~\ref {mainthm} in Section \ref {S4}. 
 Then, in Section \ref {S5}, we discuss the relation and interpretation in terms of soups of reflected Brownian motions, and to the construction of other ``reflected'' CLEs. 
We then discuss results related to the fact that the Neumann GFF is defined up to an additive constant, and we  conclude with some further comments about related work and work in progress.

Since some arguments that we will use are directly adapted from those that have been developed and used in the context of the 
Dirichlet GFF, we choose to provide the more detailed proofs only for the pivotal and novel parts. 

\section{Background: The ALE of a Dirichlet GFF}
\label {S2}
In this section, we will describe what we will refer to as the ALE\footnote{This terminology, referring among other things to Arc Loop Ensembles, had been proposed by J. Aru and Avelio S. in the course of the preparation of \cite {ASW}, and at that time the third coauthor was reluctant
to introduce such a new terminology, especially in relation with the Beatles of the free field introduced in \cite {ASW}... We finally opt here for this terminology to 
stress that the work done in \cite {ASW} was influential for the present one. 
ALE should be ideally come with a Chilean or Estonian accent.} of a GFF with Dirichlet boundary conditions. 
We survey here known facts that are part of  the general framework of level lines of the GFF as first pointed out 
in \cite {SchSh,SchSh2,Dub}, see also \cite {WangWu1,WangWu,PowellWu,Wgff} for survey and variants. For all of this section, we refer to  \cite {ASW} for background and details (further related items are discussed in \cite {AS,ALS}). 

{
Recall that the GFF with Dirichlet boundary conditions in $D$ can be viewed as a centered 
Gaussian process $\Gamma$ indexed by the set of continuous functions $f$ in $D$ with compact support in $D$. 
The covariance of this process (which therefore defines the law of $\Gamma$) is 
$$ E [ \Gamma (f) \Gamma (g) ] =  \int_{D \times D} G_D (z, z') f(z) g(z') dz dz'$$
where $G=G_D$ is the Green's function in $D$ with Dirichlet boundary conditions  (here and in the sequel, $dz$ will stand for the two-dimensional Lebesgue measure). 
Throughout the paper, we will use the normalization of Green's functions $G(z,z')$ such
that $G(z,z') \sim (-2\pi)^{-1} \log |z-z'|$ as $z' \to z$ when $z$ is in the interior of the considered domain. 
With this normalization, 
the natural height-gap as introduced by \cite {SchSh,SchSh2,Dub} is equal to
$2 \lambda$, where $\lambda = \sqrt { \pi / 8}$ (this value of $\lambda$ will be fixed throughout the paper). 

We can note that one can in fact define $\Gamma$ on a larger set of functions (or measures). For instance, the previous definition 
obviously works also for the set of continuous functions $f$ in $D$ such that 
$$  \int_{D \times D} G_D (z, z') f(z) f(z') dz dz' <  \infty $$
and we will implicitely use this extension at various instances in the paper (when we define the ALE decomposition of the GFF for instance, we will consider 
continuous bounded functions in a bounded domain $D$). 
However, when one knows the process $\Gamma$ defined on the set of 
continuous functions with compact support, one can extend $\Gamma$ to this larger set of functions by a continuity argument.
We will use the same procedure when we will consider and define the Gaussian Free Fields with different boundary conditions (Neumann, or mixed Neumann-Dirichlet) in 
a domain $D$ as random processes indexed by the set of continuous compactly supported functions in $D$. 
}

\medbreak

It is useful to first recall the Miller-Sheffield \cite {MS} coupling of CLE$_4$ with the GFF (see also \cite {ASW} for details): 
Consider a simply connected domain $D$ with non-polar boundary, and a simple non-nested CLE$_4$ in $D$. Recall that this is a random collection $(\gamma_j)_{j \in J}$ of disjoint simple loops in $D$ such that: 

(a) Each given point in $D$ is almost surely surrounded by exactly one loop of the CLE$_4$. 

(b) The law of the CLE$_4$ is invariant under any conformal automorphism of $D$. 

Note that (a) and (b) imply that the set of points that are surrounded by no loop (this set is called the CLE$_4$ carpet) has zero Lebesgue measure. 
In fact, it can be shown that its Hausdorff dimension is almost surely equal to  $15/16$ \cite {NW,SSW}. 
Once one has sampled this CLE$_4$, one can toss an independent fair coin $\eps_j \in \{ -1 , +1 \}$ for each CLE$_4$ loop, and consider for each realization of the CLE$_4$, the random function 
that is equal to $2 \lambda \eps_j$ in each domain $O_j$ encircled by $\gamma_j$, and to $0$ in $D \setminus \cup_j O_j$, i.e., the function $\sum_j 2 \lambda \eps_j 1_{O_j}$. 

Then (i.e., conditionally on the CLE$_4$), inside of each $O_j$, define an independent GFF $\Gamma_j$ with Dirichlet boundary conditions. The coupling states that the field 
$ \sum_j ( \Gamma_j + 2 \lambda \eps_j 1_{O_j} ) $ is a Dirichlet GFF in $D$. Note that when $f$ is a compactly supported function in $D$, for each $j$,
the function $f$ restricted to $O_j$ is a continuous bounded function in $O_j$, so that one can define $\Gamma_j (f 1_{O_j})$. Note that conditionally on the CLE$_4$, the series   
$$ \sum_j  ( \Gamma_j (f 1_{O_j}) + 2 \lambda \eps_j \int_{O_j} f(z) dz ) $$ 
is a series of independent random variables with zero mean, and the usual $L^2$ criterion ensures that it almost surely converges.

One important feature of this coupling, also pointed out by Miller and Sheffield is that the CLE$_4$ and the labels $\eps_j$ are in fact deterministic functions of the GFF
that they construct in this way (see again \cite {ASW}). 

Heuristically, the CLE$_4$ loops can be viewed as level lines, or more precisely as height-gap lines: The coupling shows indeed that on the inside of $\gamma_j$, the GFF behaves like a GFF with boundary conditions $2\eps_j \lambda$, while on the outside of $O_j$, $\gamma_j$ plays the role of a Dirichlet boundary condition for the GFF. The curves $(\gamma_j)$ can be therefore informally described as
$2 \lambda$ v. $0$ (or $-2 \lambda$ v. $0$) height-gap lines. They are the first set of such loops that one encounter when one starts from $\partial D$ (where the height of the GFF is $0$). In the sequel, we will say that refer to an $a$-level line for an $a+\lambda$ v. $a-\lambda$ height-gap line (so that the CLE$_4$ loops are $\lambda$ level lines, or $-\lambda$ level lines).    

The nested CLE$_4$ is then obtained by iterating the same construction for each $\Gamma_j$ and so on -- see for instance \cite {ASW} and the references therein.

\begin{figure}[ht!]
\includegraphics[scale=0.75]{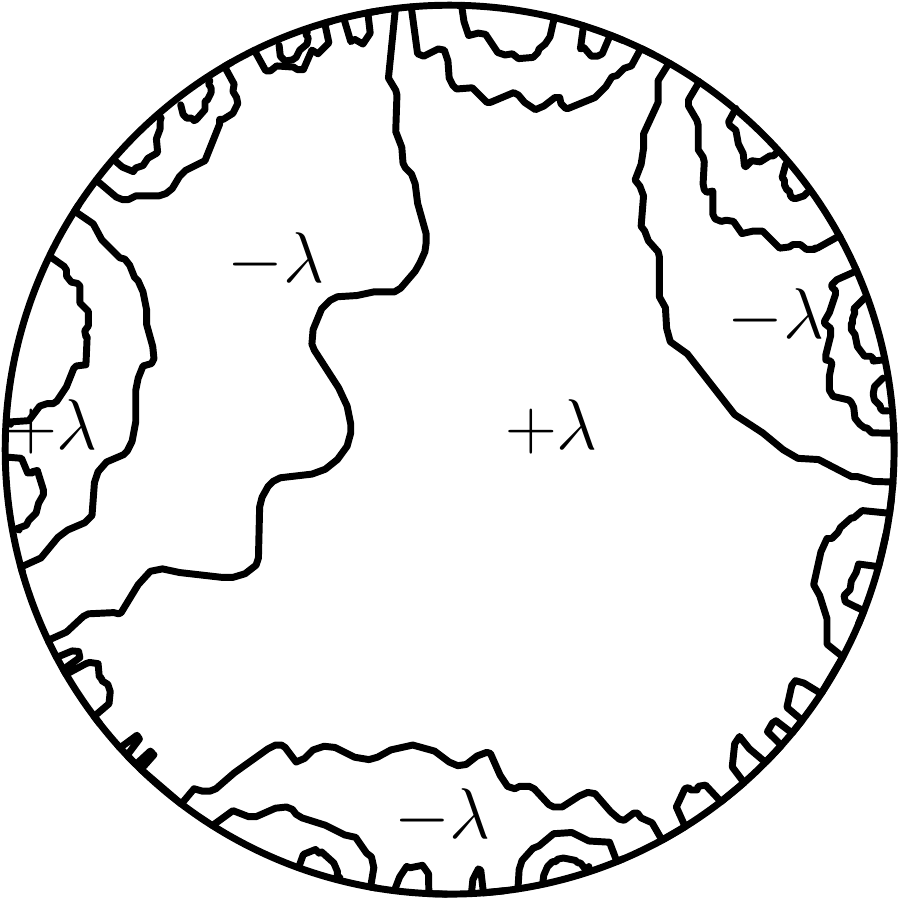}
\quad 
\includegraphics[scale=0.75]{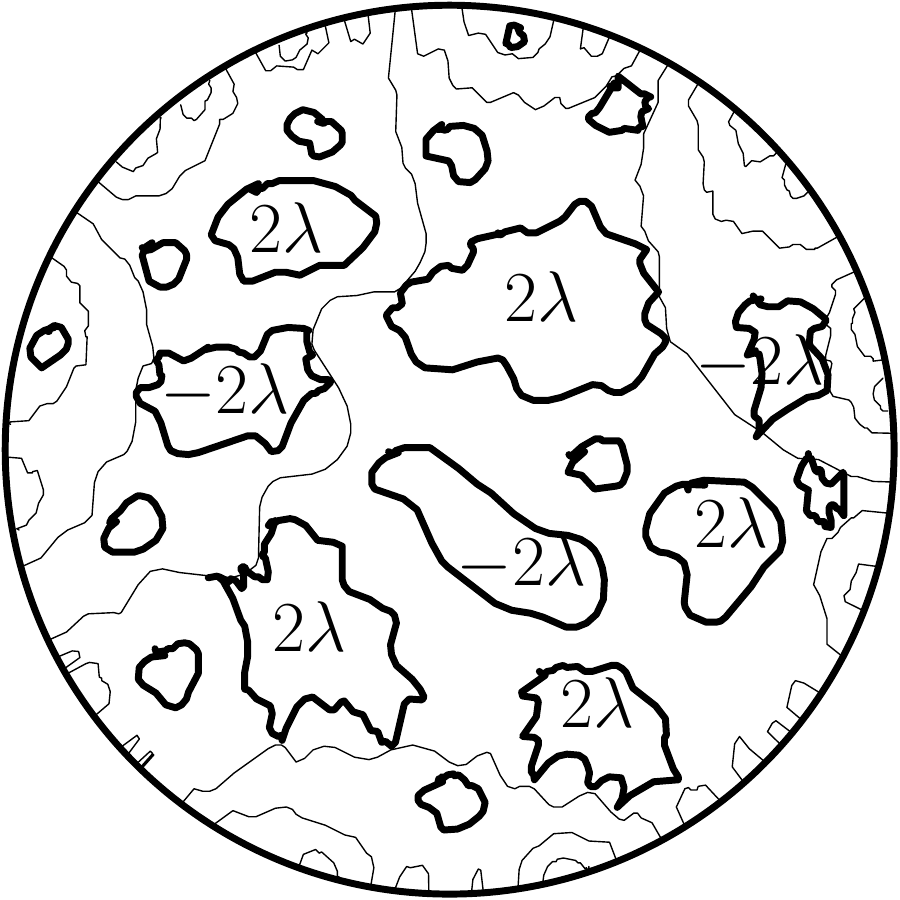}
  \caption {Structures within a Dirichlet GFF: The ALE with the alternating heights $\pm \lambda$, some of the embedded CLE$_4$ loops with their heights $\pm 2 \lambda$}
  \label{pic1}\end{figure}
\medbreak

We now describe the closely related coupling that will play a very important role in the present paper. Instead of looking at the first set of loops with height $\lambda$ or $-\lambda$ that one encounters when starting from $\partial D$ (this is the CLE$_4$ depicted on the right-hand side of Figure \ref {pic1}), 
it turns out to be possible to make sense rigorously of the collection of all first level lines of height $0$ that one encounters when one starts from the boundary -- these are 
the level lines at height $0$ that do intersect $\partial D$ (depicted on the left-hand side of Figure \ref {pic1}). 
As explained in \cite {ASW}, it is in some sense the most natural and sparsest level-hitting set that one can define in a Dirichlet GFF. 
This set $A$ can be defined by a branching family of SLE$_4 (-1, -1)$ processes, and is denoted by $A_{-\lambda, \lambda}$ in \cite {ASW} -- it is also one very special particular case of the boundary conformal loop ensembles defined in \cite {MSW} (and has also appeared before, see \cite {SchSh2,WangWu}). This is what we will call an ALE in the present paper. Let us first summarize its main properties: 

(a) An ALE is formed by a countable union of disjoint SLE$_4$-type arcs in $D$ joining two boundary points of $D$. It is locally finite (each compact subset of $D$ intersects only finitely many ALE arcs). 

(b) The law of the ALE is invariant under the group of conformal automorphisms of $D$.

A consequence of these items is that the Hausdorff dimension of the ALE is that of individual SLE$_4$ curves i.e., $3/2$ (see \cite {Be,RS}).  
Furthermore, each ALE arc will be a shared boundary between two adjacent connected components of the complement $O$ of the ALE in $D$. In particular, the family of connected components 
of $O$ comes naturally equipped with a tree-like structure given by the adjacency relation.

We will sometimes refer to the connected components $O_j$ as the cells of the ALE. The construction via SLE$_4 (-1)$ shows that the dimension of the intersection of the boundary of a cell with the boundary of $D$ is equal to 
$3/4 $ when the boundary of $D$ is smooth (using for instance Theorem 1.6 of \cite {MWu}). 

\medbreak 

One can then couple a GFF with Dirichlet boundary conditions with an ALE  as follows: Toss first one fair coin $\eps \in \{ -1 , 1 \}$ in order to decide the label of the connected component of $O$  that contains the origin, and then define deterministically the labels $\eps_j$ of all the other components $O_j$ of $O$ in such a way that any two adjacent components have opposite labels (with obvious terminology, $\eps_j$ will be equal to $\eps$ on the even components $O_j$ and 
to $- \eps$ on the odd ones). Define also an independent GFF $\Gamma_j$ with Dirichlet boundary conditions in each $O_j$. Then, it turns out that the field 
$$ \Gamma := \sum_j ( \Gamma_j + \lambda \eps_j 1_{O_j} ) $$ is a GFF with Dirichlet boundary conditions in $D$. 
Again (see \cite {ASW} and the references therein for all these facts), the ALE and its labels are  in fact a deterministic function of  $\Gamma$ (we can therefore call 
``the'' ALE $A=A (\Gamma)$  of $\Gamma$). One way to characterize $A$ as a deterministic function of $\Gamma$ is to say that it is the collection
of all the $0$-level-lines of $\Gamma$ that touch the boundary of $D$.

Note that there are two ways to think of the ALE. For each connected component $O_j$, the outer boundary of $O_j$ will consist of the concatenation of ALE arcs, and form one single level line that is bouncing off from the boundary of $D$. In other words, one could also view the ALE as a collection of loops. But, in view of the next section,  we choose to define the ALE as a collection of disjoint level-arcs (i.e. level lines inside the domain that start and end on the boundary).

As explained in \cite {ASW}, starting from an ALE, it is easy to iterate the procedure by considering the ALE of each $\Gamma_j$ and so on, until one finds the first $\pm \lambda$ level lines: 
One can then  view the CLE$_4$ as an iterated nested ALE (see again Figure \ref {pic1} for a sketch). Note that the number of iterations needed before discovering the CLE$_4$ loop that surrounds a given point is random (which explains why the dimension of the CLE$_4$ carpet is larger than $3/2$).

\section {The  coupling of the two fields: Statement and consequences}
\label {S3}

Recall that the  Neumann GFF $\Theta$ in a domain $D$ is a similar (rather rough) conformally invariant centered Gaussian 
random field as the GFF with Dirichlet boundary conditions,
but that it is defined only up to an additive constant. In other words, one defines only its gradient, or equivalently, 
it is possible to make sense of the centered Gaussian random variable $\Theta (f)$ when $f$ is a continuous function with compact support 
such that $\int_D f(z) dz = 0 $, but one can not talk for instance of $\Theta (1)$ or of the mean value of $\Theta$ on a ball. 

Since we will only consider the Neumann GFF in simply connected domains, one handy quick way to define it is to first define it in the upper half-plane and to then use 
conformal invariance to extend the definition to an simply connected domain $D$. 
In the upper half-plane $\HH$, the Neumann GFF $\Theta$
is the centered Gaussian process $(\Theta (f))$, defined on the set of continuous functions $f$ with mean $0$ with compact support in $\HH$,
with covariance given by
$$ E [ \Theta (f) \Theta (g) ] = \int_{\HH \times \HH} f(z) g(z') G_\HH^{\mathcal N} (z, z') dz dz'.$$ 
The ``Neumann Green's function''  $G_\HH^{\mathcal N}$ that we use here is defined as 
$$ G_\HH^{\mathcal N} (z,z') :=  (2 \pi)^{-1} (\log |x-y| + \log | x - \overline y| ).$$ 

On the other hand, if we are actually given a scalar field (i.e., a real-valued process $(\Lambda(f))$ 
defined on the set of smooth functions supported on compact subsets of $D$
and that is linear with respect to $f$), we can call it a realization of the Neumann GFF if it behaves the same way as
the Neumann GFF $\Theta$ when one restricts it to the set of  functions with zero mean 
(then, adding any constant function to $\Lambda$ would indeed provide another realization of the same Neumann GFF $\Theta$).

We are now almost ready to explain our coupling of the Dirichlet GFF $\Gamma$ with the Neumann GFF: Consider first 
the decomposition of $\Gamma$ with the ALE $A$, the fields $\Gamma_j$ and the labels $\eps_j$ as described in Section \ref {S2}. We then also define
a simple random walk indexed by the tree structure of the ALE. More precisely, we define a new set of labels $\eta_j$ with odd values iteratively as follows: 
The label of the connected component that contains the origin is chosen to be $\eps$, and for each connected component adjacent to this
first connected component, one tosses an independent fair coin to decide if its label is  $\eps + 2$ or $\eps -2$. Then iteratively, 
toss another independent fair coin for each  connected component  adjacent to the previously labelled components to decide whether their labels differ by $+2$ or $-2$. 
In other words, when one follows a path of adjacent connected components that go to the boundary, instead of seeing the alternating $\eps_j$ labels $\eps$, $-\eps$, $\eps$ etc., one now sees a simple random walk with jumps of $\pm 2$. We call this new set of labels $(\eta_j)$. See Figure~\ref {pic00} and Figure~\ref {pic2} for an illustration (with labels multiplied by $\lambda$). Note that with this definition, $\eta_j - \eps_j \in 4 \Z$ for all $j$.

\begin{figure}[ht!]
\null
\vskip 5mm
\includegraphics[scale=0.8]{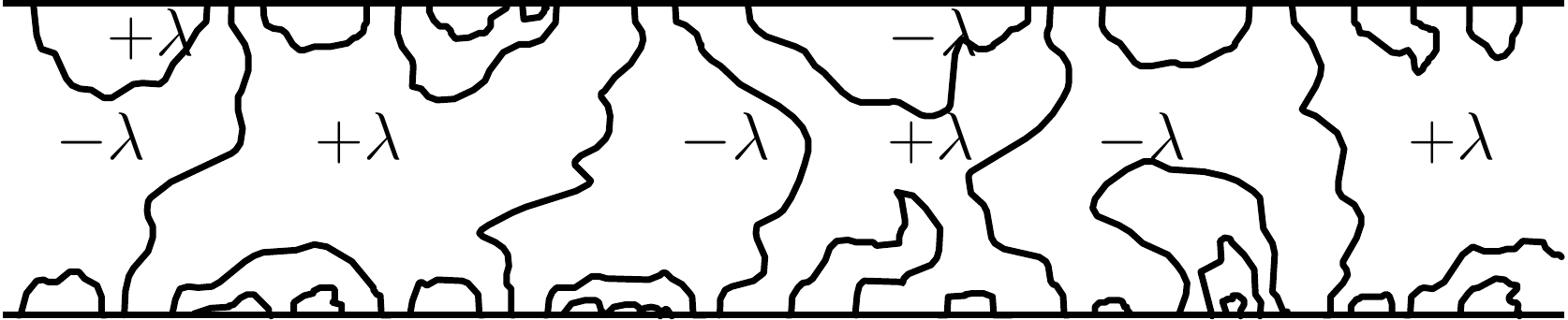}
\vskip 5mm
\includegraphics[scale=0.8]{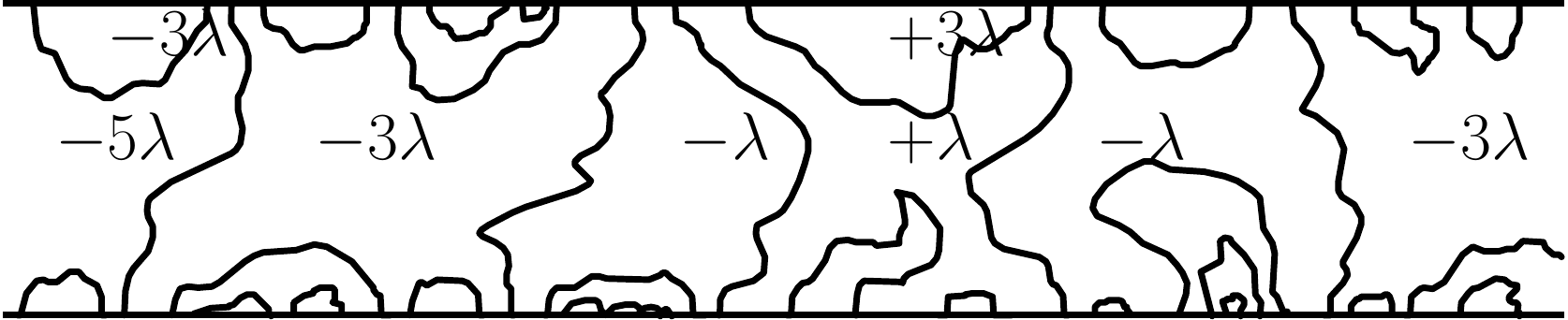}
  \caption {The analogue of Figure \ref {pic00} in a strip:
  Illustration of the Dirichlet and Neumann  coupling.  
  Top: The ALE with the corresponding alternating heights. Adding to those heights independent GFFs in each tile does construct a GFF with Dirichlet boundary conditions in the strip. 
  Bottom: The same ALE  but with the randomly chosen signs for height-gaps. 
  Adding to those heights independent GFFs in each tile does construct a GFF with Neumann boundary conditions in the strip.}
  \label{pic2}\end{figure}

\begin {theorem}[The Neumann GFF - Dirichlet GFF coupling]
\label {mainthm}
The field $\Lambda := \sum_j (\Gamma_j + \lambda \eta_j 1_{O_j})$ is a realization of a GFF with Neumann boundary conditions. \end {theorem}

Recall that the field $\Gamma = 
\sum_j (\Gamma_j +  \lambda \eps_j 1_{O_j})$ is a GFF with Dirichlet boundary conditions. This theorem therefore provides
a coupling of the GFF $\Lambda$ with Neumann boundary conditions with the GFF $\Gamma$, such that $ \Lambda- \Gamma$ is the function that is exactly equal to $\lambda (\eta_j- 1)$ on the even connected components $O_j$ of the ALE of $\Gamma$, and to $ \lambda (\eta_j +1 )$ on the odd ones.

Note again that there is no problem in the theorem in order to define $\Lambda (f)$ for continuous functions with compact support in $D$ because almost surely, the support of $f$ intersects only finitely many $O_j$'s. 

It is important to stress that this coupling is very different from the coupling obtained by first reading off the boundary values of the Neumann GFF i.e., the harmonic extension to $D$ of these boundary values, and  to view the Neumann GFF as the sum of this random harmonic function (that is defined up to additive constants, just as the Neumann GFF is) with an independent Dirichlet GFF. In the coupling described in Theorem \ref {mainthm}, the difference $\Lambda - \Gamma$ is not a continuous function in $D$, and furthermore
the Dirichlet GFF $\Gamma$ and the difference $\Lambda -\Gamma$ are very strongly correlated (the discontinuity lines of the latter are the boundary-touching level lines of the former).

Note that in the construction of Theorem \ref {mainthm}, the law of  $\Lambda$ is indeed conformally invariant if one views it as defined up to constants (as the special role played by the origin in the construction disappears): The image of $\Lambda$ via a given conformal automorphism of $D$ that maps some point $z_0$ onto the origin is distributed 
exactly like $\Lambda$ shifted by a random additive 
constant field with values in $2 \lambda \Z$ corresponding to the height of the ALE cell for $\Lambda$ that contained $z_0$.

It is worthwhile comparing the collection of level lines of $\Gamma$ with those of $\Lambda$, as some aspects can appear confusing at first sight:  

On the one hand, Theorem \ref {mainthm} implies that:

- The family ${\mathcal A}$ of all boundary-touching level arcs at levels in $2 \lambda \Z$ (here and in the sequel, this means that we look at the union of these level arcs, but that we do not record the actual value of their height in $2\lambda \Z$) for $\Lambda$ and for $\Gamma$ do coincide. Note  however that the height of such an arc for $\Gamma$ 
is necessarily $0$ for $\Gamma$ but can be non-zero for $\Lambda$. 

- The collection of all other (i.e., non-boundary touching) level lines with height in $2 \lambda \Z$ (both for $\Gamma$ and $\Lambda$) which correspond to the nested CLE$_4$ loops for the  $\Gamma_j$'s are therefore the same. 
Hence, we can say that the whole collections of level arcs and loops with height in $2 \lambda \Z$ do coincide for $\Gamma$ and $\Lambda$ (here again, we look at the union of these level-lines but do not record their actual height). 
We can recall (see \cite {ASW} and the references therein) that one can reconstruct deterministically the Dirichlet GFF from its collection of nested ALEs together with the knowledge of the signs of the corresponding height-gaps (there is one random coin-toss per individual ALE in the nested ALE). 
{ Since $\Lambda$ restricted to each cell $O_j$ is just $\Gamma_j + \lambda\eta_j$, one can iterate the procedure for each 
Dirichlet GFF $\Gamma_j$ and eventually, just as for the Dirichlet GFF $\Gamma$,}
reconstruct $\Lambda$  via its collection of nested level lines (and the data about the random signs of the height-gaps).

- The coupling also shows that each level loop of $\Gamma$ that does not intersect the ALE will correspond to a level loop of $\Lambda$ that does not intersect the ALE, and vice-versa. The difference between the height of such loops for $\Lambda$ and for $\Gamma$ will be in $4\lambda \Z$.

It is worthwhile to note that for the level lines that do intersect the ALE, the story is different. 
Indeed, the way in which they bounce on the ALE arcs will depend on the   
sign of the height-gaps on those ALE arcs as illustrated on Figure \ref {RR} (this is very much related to how flow lines for the GFF interact, as discussed for instance in \cite {MS1}). 
As we shall explain in Section \ref {Sshift}, this will imply for instance that there exist level-arcs at level in $\lambda + 2 \lambda \Z$ for $\Lambda$ that do hit the boundary of the domain (see again Figure \ref {RR}), while this is known not to be the case for $\Gamma$. 

\begin{figure}[ht!]
\includegraphics[scale=0.7]{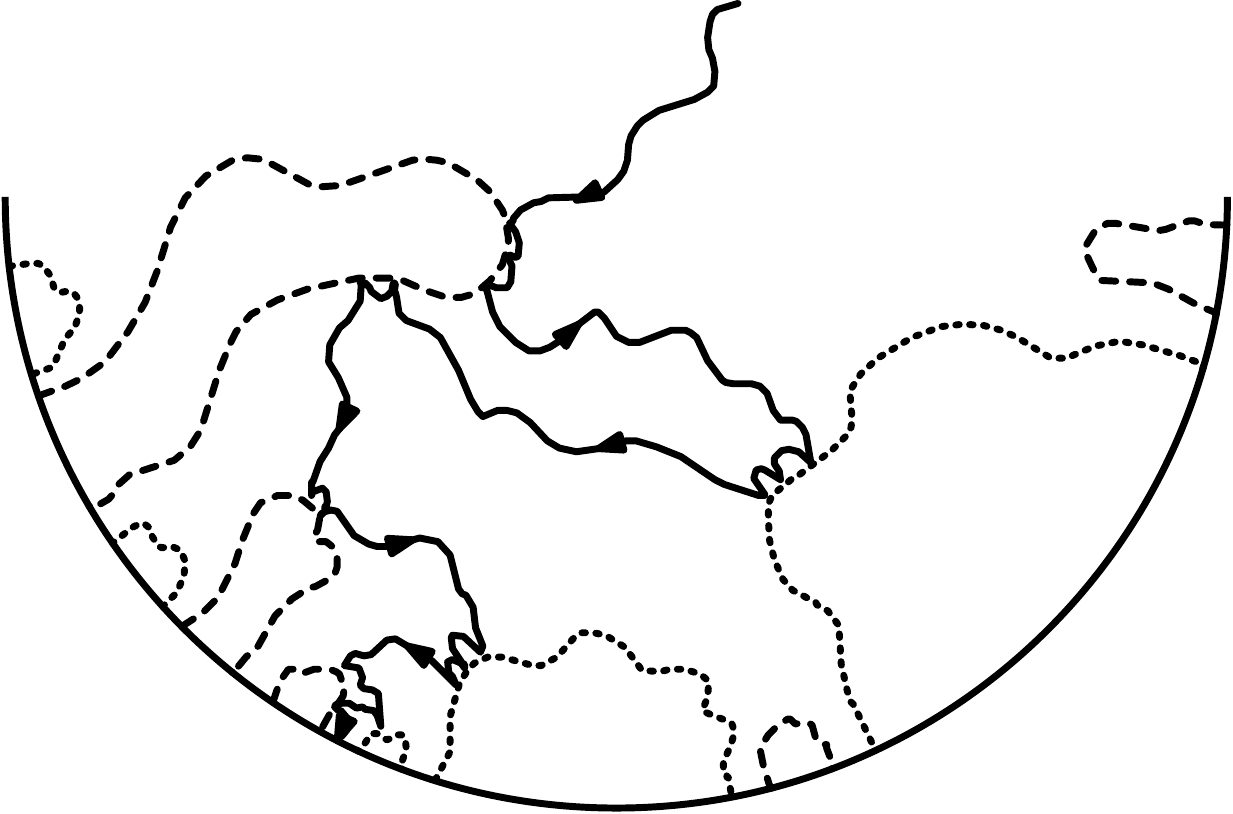}
  \caption {The $-\lambda$ level line for $\Lambda$ turns left on the ALE arcs with positive height-gap (in dotted) and right on the ALE arcs with negative height-gap (in dashed), and ends up hitting the boundary.}
  \label{RR}\end{figure}

In the same spirit, it is also easy to see that for the Dirichlet GFF $\Gamma$, the collections of boundary-touching ALE with height in $2 \lambda \Z$ (i.e. the collection of boundary-touching level lines with height $0$) have a different distribution than the collection of boundary touching level lines with height in $a + 2 \lambda \Z$ for a given $a$ in $(0, 2\lambda)$ (this follows for instance from the fact that the Hausdorff dimension of the intersection 
of the boundary of a cell defined by these level lines with the boundary of the domain is explicitely known and depends on $a$). 
On the other hand, the Neumann GFF $\Theta$ is defined ``up to an additive constant'' so that one can wonder whether, when one couples $\Theta$ with its particular realization $\Lambda$ as in Theorem \ref {mainthm}, the law of the 
collection of boundary-touching level lines of $\Lambda$ with height in $a + 2 \lambda \Z$ is independent of the value of $a$ or not. The answer will be given by the following result: 

\begin{proposition}
\label {propshift}
The law of the collection ${\mathcal A}_a$ of all the boundary-touching level lines of $\Lambda$ with height in $a + 2 \lambda \Z$ is identical to the law of ${\mathcal A}_0$ which is an ALE. Furthermore, conditionally on ${\mathcal A}_a$, the signs of the height-gaps on each of the arcs are again i.i.d., just as for ${\mathcal A}_0$ itself. 
\end{proposition}

Informally, this means that if one is given a GFF $\Theta$ with Neumann boundary conditions, then one can choose a ``reference height'' uniformly in $[0, 2 \lambda)$ in order to determine the realization of the corresponding ALE. In other words, the ALE is a deterministic function of the Dirichlet GFF, but in the case of the Neumann GFF, there is a one-parameter family of possible ALEs associated to it. A more precise related statement that follows 
from this proposition goes as follows: Suppose that
$\Theta$ is a Neumann GFF and that $\bar \Lambda$ is a realization of it. Then, if one chooses an independent uniform random variable $\zeta$ in $[0, 2\lambda]$, the field $\bar \Lambda + \zeta$ is another realization of $\Theta$ such that the 
law of the coupling of  $\Theta$ with the boundary-touching level arcs of $\bar \Lambda + \zeta$ at height in $2 \lambda \Z$ and the corresponding height gaps, is the same as the one given by the coupling of the Neumann GFF with an ALE and i.i.d. gaps as given by Theorem \ref {mainthm}. 

This is similar to the question of defining the nested CLE$_4$ as a function of the GFF in the Riemann sphere (see \cite {KW}) which is also defined up to an additive constant. 

\medbreak

These boundary touching arcs for the Neumann GFF play a similar role as the CLE$_4$ loops play for the Dirichlet GFF (or for the Neumann GFF for the inside loops). Indeed, if one views the Neumann GFF as defined on the upper-half-plane and symmetrizes the picture with respect to the real axis (one considers the union 
of the figure in the upper half-plane with its symmetric image) which is a very natural thing to do for such Neumann boundary conditions, then each arc of the ALE will hook-up with 
its symmetric image and create a loop. The story is then just as for CLE$_4$: One has a collection of disjoint simple and nested loops, and each loop 
tosses an independent fair coin in order to decide whether it is an upward or a downward $\pm 2 \lambda$  jump of the GFF. The union of the ALE and of the inside CLE$_4$ loops that one can define in the complement of the ALE are then a way to describe (and reconstruct, when one adds all the randomly chosen labels) the Neumann GFF. This is consistent with the informal description of the Neumann GFF in the 
upper half-plane to be 
the GFF in the entire plane, conditioned to be symmetric with respect to the real line, and restricted to the upper half-plane.
See Figure \ref {pic10} in the analogous case of mixed boundary conditions.  
The dimension $3/4$ of the intersection of an SLE$_4 (-1)$ path with the real line can then be interpreted as the dimension of the part of the real line that belongs to a Neumann CLE$_4$ carpet.

\section {Proof of the coupling}
\label {S4}
The main goal of this section is to prove Theorem \ref {mainthm} (Proposition \ref {propshift} will be derived in Section \ref {Sshift}). 

Let us first briefly give the definition of the GFF with mixed Dirichlet-Neumann boundary conditions that will play an important part in the rest of this paper. 
The Dirichlet GFF is associated with Brownian motion that is killed when it reaches the boundary of $D$. If one divides the boundary of $D$ into two parts, one 
part $\partial_{\mathcal D}$ where the 
Brownian motion is killed, and one part $\partial_{\mathcal N}$ where it is reflected, then it is easy to define (for instance by conformal invariance) the law of Brownian motion reflected on the latter 
part of the boundary and killed when it reaches the former, and the corresponding Dirichlet-Neumann Green's function. One can for instance first define 
this function in some well-chosen reference domain, and then generalize it by conformal invariance. 
Here, it is convenient to consider the positive quadrant ${\mathcal Q} = (0, \infty)^2$, and to define the 
mixed Dirichet-Neumann Green's function with Dirichlet conditions on $\R_+$ and Neumann 
conditions on $i\R_+$ by 
$$ G_{\mathcal Q}^{\mathcal {ND}} (z, z') = G_\HH (z, z') + G_\HH (- \bar z,  {z'}).$$
Then, one can define the Neumann-Dirichlet GFF in ${\mathcal Q}$ in exactly the same way as the Dirichlet GFF using this Green's function instead. 

When the boundary of $D$ is smooth and $H$ is a bounded harmonic function in $D$ such that the normal derivative on $\partial_{\mathcal N}$ vanishes, we define 
the GFF in $D$ with boundary conditions equal to those of $H$ on $\partial_{\mathcal D}$ and Neumann boundary conditions on $\partial_{\mathcal N}$ (we will refer to those conditions as mixed boundary conditions) to be the sum 
of the previous Neumann-Dirichlet GFF with $H$. Again, one can then extent this definition to domains with non-smooth boundaries by conformal invariance. 

\medbreak

We are now ready to describe the following warm-up to the proof of Theorem \ref {mainthm} : 

\begin {itemize}
 \item Let us choose $w<o$ on the real line.  When one considers a GFF $\Gamma$ in the upper half-plane with boundary conditions $-\lambda$ on $(-\infty, w)$, $\lambda$ on $(w, o)$ and $0$ on $(o, \infty)$, then one can trace the 
 $0$-level line $\gamma$ that emanates from $w$ (with $-\lambda$ boundary conditions on its left, and $\lambda$ boundary conditions on its right) as illustrated on Figure \ref {pic4}. Given that when one looks at the GFF away from $(o, \infty)$, it will be absolutely continuous with respect to the GFF with boundary conditions $-\lambda$ on $(-\infty, w)$ and $+ \lambda$ on $(w,\infty)$, the existence and  uniqueness of $\gamma$ until the time $\tau$ at which it hits $(o, \infty)$ can be viewed as a consequence of the corresponding known fact in the latter boundary conditions. 
 In fact, the conformal Markov property with one additional marked point, and the characterization of SLE$_4 (\rho)$ processes as the only continuous processes that satisfies this property
 imply immediately that the law of $\gamma$ up to $\tau$ is that of an SLE$_4 (\rho)$ for some value of $\rho$. The actual value of $\rho$ is necessarily equal to $-1$, because in the previous setup, $\infty$ and $o$ play somehow symmetric roles: $\gamma$ is an SLE$_4 (\rho)$ from $w$ to $o$ with marked point at infinity for the same value of $\rho$. This implies by the standard SLE coordinate change formulas (see for instance \cite {SchrammWilson}) that $\kappa -6 -\rho = \rho$ with 
 $\kappa =4$, i.e., $\rho = -1$. Of course, the previous lines are not the shortest derivation of the known fact that $\eta$ is an SLE$_4 (-1)$ but we highlighted this argument in order 
 to explain why the same result will hold in the next item. 
 \begin{figure}[ht!]
\includegraphics[scale=0.6]{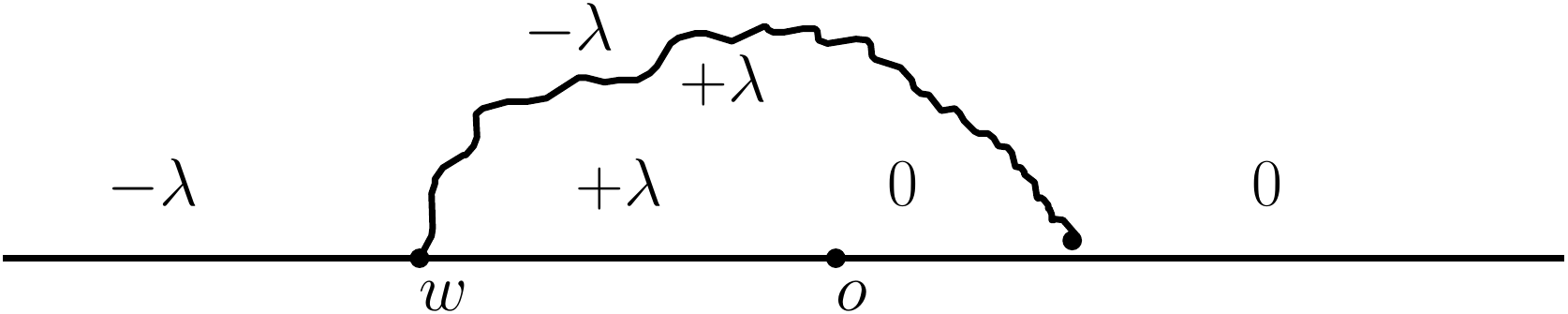}
  \caption {The SLE$_4 (-1)$ corresponding to the $0$-level line of the GFF with these boundary conditions}
  \label{pic4}\end{figure}

 \item Let us now consider instead a GFF $\tilde \Gamma$ in the upper half-plane, with boundary conditions $-\lambda$ on $(-\infty, w)$, $\lambda$ on $(w, o)$, but this time, on $(o,\infty)$, one takes Neumann boundary conditions. For the same absolute continuity reasons, one can define the $0$-level line $\tilde \gamma$ that emanates from $w$ (see Figure \ref {pic5}) up to the first time $\tilde \tau$ at which it will hit $(o, \infty)$, and until that time, it will be an SLE$_4 (\rho)$ curve from $w$ to $\infty$ with marked point at $o$
 using the same argument based on the conformal Markov property. Again, $\infty$ and $o$ play symmetric roles, so that it should also be an SLE$_4 (\rho)$ curve from $w$ to $o$ with marked point at $\infty$ for the same value of $\rho$, which shows that $\tilde \gamma$ is also an SLE$_4 (-1)$ curve
 (this is also not new, see for instance Izyurov-Kyt\"ol\"a \cite {IK} -- we wrote out this little argument to highlight the basic reason for which the same SLE$_4 (-1)$ path 
 appears in our two settings).
 
 \begin{figure}[ht!]
\includegraphics[scale=0.6]{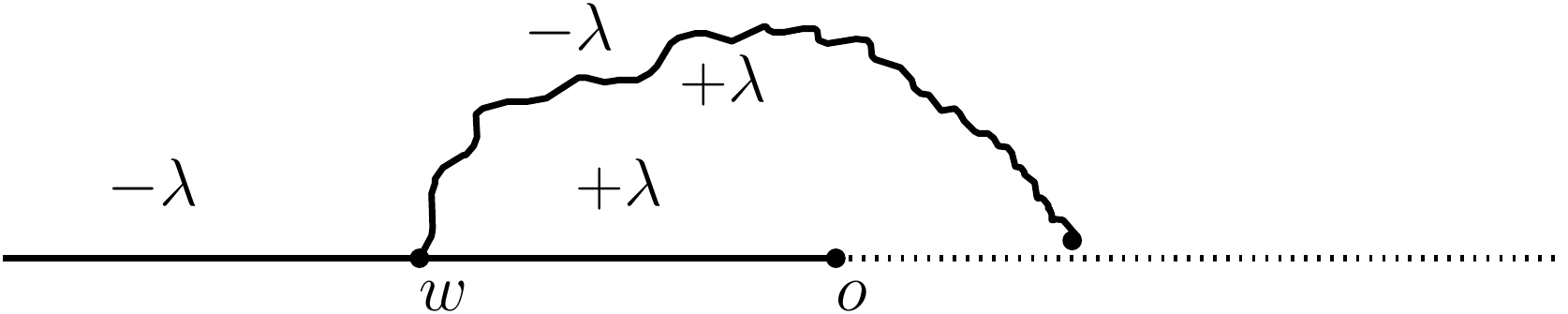}
  \caption {The same SLE$_4 (-1)$, this time corresponding to the $0$-level line of a GFF with these new mixed boundary conditions (Neumann instead of $0$ on $(o, \infty)$).}
  \label{pic5}\end{figure}

 \end {itemize}
 
Here, we can already couple these GFFs $\Gamma$ and $\tilde \Gamma$ with these different boundary conditions in such a way that $\gamma$ and $\tilde \gamma$ coincide up to 
their hitting time of $(o, \infty)$. The main question is now to see what one can do beyond the time $\tau$. 
Note first that in the case of $\Gamma$, it is possible to continue the branching SLE$_4 (-1)$ as $0$-level lines and to trace all the arcs of the ALE associated to $\Gamma$ via the SLE$_4 (-1)$ branching tree \`a la \cite {Sh}. The question is what to do in the case of $\tilde \Gamma$, when one has mixed boundary conditions.

By symmetry, it will suffice to describe how to continue to trace $\tilde \gamma$ targeting infinity (the other branches will be defined similarly). At time $\tilde \tau$, the 
boundary conditions on the unbounded connected component of the complement of $\tilde \gamma$ are $-\lambda$ on the left of the tip of the curve, and Neumann on $(\tilde \gamma (\tilde \tau), \infty)$. By conformal invariance, this means that we are looking at the case of a GFF in the upper half-plane with boundary conditions 
$-\lambda$ on $\R_-$ and Neumann on $\R_+$ (and that we want to couple it with a GFF in the upper half-plane with boundary conditions $-\lambda$ on $\R_-$ and $0$ on $\R_+$). 
Note first that at every point, the expected value of this GFF (with mixed boundary conditions) is  $-\lambda$. 

Here, the new input will be to notice that it is possible to directly adapt what has been done for instance for the definition of the CLE$_4$-GFF coupling out of the side-swapping SLE$_4 (-2)$ branching tree (see for instance \cite {Sh,WW,ASW,Wgff})
to the present setting:
If we were to insert a small interval $(0, \eps)$ with boundary height $+ \lambda$, or with $-3 \lambda$  between the $-\lambda$ and the Neumann boundary parts,
then our previous analysis shows how to continue the interface for a little while, until it hits the Neumann boundary arc. Furthermore, if one tosses a fair coin to decide whether one inserts $+ \lambda$ or $-3 \lambda$, one will preserve the fact that the mean value is $-\lambda$ i.e., the martingale property of the height-function, and in fact the fluctuations of the height will be 
symmetric in both cases. Iterating this procedure and letting $\eps \to 0$ would provide the desired coupling.

 \begin{figure}[ht!]
\includegraphics[scale=0.6]{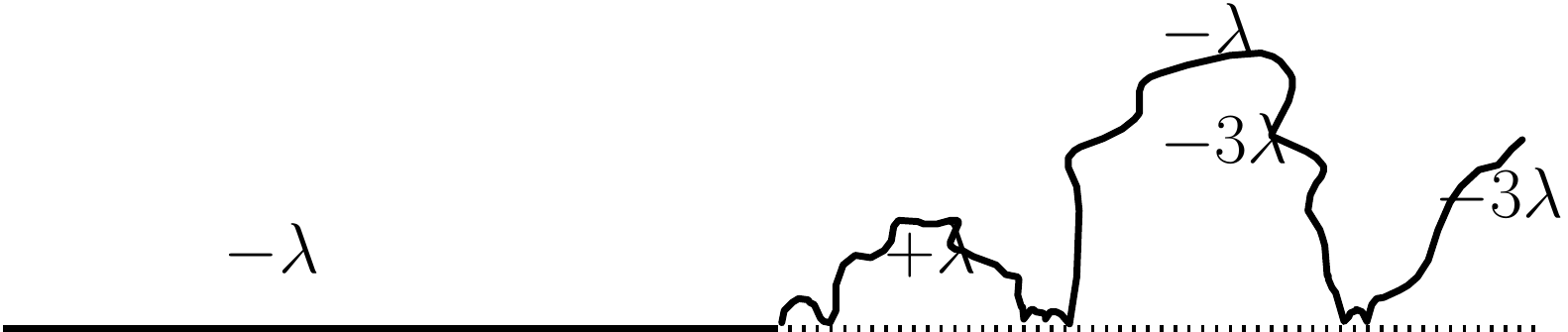}
  \caption {The SLE$_4 (-1)$ with randomly chosen heights at each excursion.}
  \label{pic6}\end{figure}

After this warm-up, let us turn to the actual proof.  
Our first goal will now be to couple a GFF with boundary conditions $-\lambda$ on $\R_-$ and $0$ on $\R_+$ with a GFF with 
 boundary conditions $-\lambda$ on $\R_-$ and Neumann on $\R_+$. 
Based on the previous considerations,  the strategy of the coupling (illustrated in Figure \ref {pic6}) now goes as follows: 
\begin {itemize}
\item Sample an SLE$_4 (-1)$  from $0$ to $\infty$ that we denote by $\eta$. This is known to be a simple continuous curve from $0$ to infinity, that touches $\R_+$ on a fractal set of points, and does not intersect $\R_-$. It can be therefore viewed as the concatenation of disjoint excursions away from the positive real line. 
\item Inside each of the bounded connected component of $\HH \setminus \eta$ (each excursion of $\eta$ corresponds to one such connected component),
toss an independent fair coin in order to decide whether to choose $+ \lambda$ or $-3 \lambda$ boundary conditions on the part of the boundary of that component traced by the corresponding excursion of $\eta$, and keep Neumann boundary conditions on $\R_+$. On the unbounded connected component of 
the complement of $\eta$, choose uniform $-\lambda$ boundary conditions. This includes $\R_-$ and the ``left-side'' of the entire curve 
$\eta$ (see Figure \ref {pic6}). 
\item Then inside each of the connected components of the complement of $\eta$, sample an independent GFF with the above-mentioned boundary conditions i.e. the sum of a $0$-boundary GFF with 
the harmonic function $h$ that has these boundary conditions. 
\end {itemize} 
Our first step is to check that:
\begin {lemma}
\label {l3}
 The obtained field $\hat \Gamma$ is indeed a GFF in $\HH$ with boundary conditions given by $-\lambda$ on $\R_-$ and Neumann on $\R_+$. 
\end {lemma} 
\begin {proof}
The proof of this lemma goes along the same lines as for the radial SLE$_4 (-2)$ coupling with the GFF as for instance outlined in \cite {ASW}.
First, one notices that by properties of Gaussian processes, it is sufficient 
to show that the law of $\hat \Gamma (f)$ for every individual continuous function $f$ with compact support in the upper half-plane 
is Gaussian with the right mean and variance.

So, we fix such a function $f$ with compact support in $\HH$. 
Then, for each excursion of $\eta$ away from the real line, we toss an independent $\pm 1$ coin, and for every time $t$, we define the harmonic function $h_t$ in the complement of $\eta [0,t]$ via the following boundary conditions: 
Neumann boundary conditions on $\R_+$, constant boundary conditions equal to $-\lambda$ on $\R_-$ as well as on the left-hand side of $\eta [0,t]$,
and  constant boundary conditions 
$\lambda$ or $-3 \lambda$ on the right-hand side of $\eta[0,t]$, depending on the sign of the coin-tossing of the corresponding excursion of $\eta$.
We denote by ${\mathcal F}_t$ the $\sigma$-field generated by $\eta [0,t]$ and the coin-tosses that occured before time $t$ (i.e., corresponding to all excursions of $\eta$ that started before time $t$). 

If we are given a point $z$ in the upper half-plane, we see that $t \mapsto h_t (z)$ will evolve like a continuous local martingale during the excursion intervals of $\eta$. This is exactly 
a consequence of the previous observation about SLE$_4 (-1)$ processes coupled with mixed boundary conditions. Note also that by definition, 
$h_t( z)$ remains in the interval $[-3 \lambda, \lambda]$ and converges as $t \to \infty$ to $h(z)$. Let us define 
$$ M_t := \int f(z) h_t(z) dz .$$
By dominated convergence, it converges almost surely and in $L^1$ to $\int f(z) h(z) dz$, and furthermore, for every given $t > 0$, 
if $\tau_t$ denotes the first time after $t$ at which $\eta$ hits the real line, then $s \mapsto M_{\min(t+s, \tau_t)}$ will be a continuous bounded martingale. 
Furthermore, the quadratic variation of this martingale on this interval
is exactly compensating the decrease of the Green's functions in the complement of $\eta$ (this is exactly the integrated version of
Proposition 4 of Izyurov-Kyt\"ol\"a \cite {IK}, see 
also Sheffield \cite {SheffieldQZ} for the Dirichlet version): 
$$ d\langle M \rangle_t = - d \int f(x) f(y) G_{\HH \setminus \eta [0,t]}^{\mathcal {ND}} (x,y) dx dy$$
where $G^{\mathcal {ND}}$ denotes the Green's function with mixed Dirichlet-Neumann boundary conditions (Neumann on $\R_-$ and Dirichlet elsewhere).
The first additional item to check is that the martingale property is preserved after $\tau_t$, which follows directly from the coin tossing procedure (indeed, if $\tau_t < t+s$, then the conditional expectation of $M_{t+s}$ given ${\mathcal F_{\tau_t}}$ will be exactly $M_{\tau_t}$ -- to see this we can first discover $\eta[\tau_t, t +s]$ and then note that the contributions due to the coin tosses that occur during $[\tau_t, t +s]$ are symmetric and cancel out). In other words, for all positive $s$, we have that 
$ M_t = E [ M_{t+s} | {\mathcal F}_t ]$, and letting $s \to \infty$, $M_t = E [ M_\infty | {\mathcal F}_t]$. 

The second additional item to check is the $t \mapsto M_t$ is in fact continuous on all of $\R_+$. This follows directly from the definition of $M_t$ and the fact that $t \mapsto \eta_t$ is known 
to be a continuous curve. Hence, we now know that $t \mapsto M_t$ is a continuous bounded martingale, i.e., a Brownian motion time-changed by its quadratic variation. 

In order to conclude, we just need to justify the fact that 
$$ \langle M \rangle_\infty = \int f(x) f(y) ( G_{\HH}^{\mathcal {ND}} (x,y) - G_{\HH \setminus \eta}^{\mathcal {ND}} (x,y) )  dx dy.$$ 
Indeed, if this holds, then we can conclude just as in \cite {SheffieldQZ}. 
So, what remains to be checked is that the measure $d\langle M \rangle_t$ puts zero mass on the closed set of times at which $\eta$ hits the boundary. 
Here, we can use the fact that the definition of $M$ shows for any stopping times $\sigma < \sigma'$ such that $\eta_\sigma$ and $\eta_{\sigma'}$ 
lie on the real line, one has $M_{\sigma} = M_{\sigma'}$ unless $\eta$ made an excursion into the support of $f$ between $\sigma$ and $\sigma'$, and we know (because 
$\eta$ is continuous and the compact support of $f$ is at positive distance from the boundary) that almost surely, only finitely many excursions of $\eta$ do make it into the support of $f$. 
\end {proof}

The second step in the proof of the theorem is to notice that once all of $\eta$ has been traced, it is possible to iterate this coupling inside each of the bounded connected components 
of the complement of $\eta$ (indeed, the boundary conditions there are again constant on one arc, and Neumann on the real segment). More precisely, we define 
the random function $h_n (z)$  obtained after $n$ iterations of ``layers'' of SLE$_4 (-1)$ processes with the random height-gaps. The function $h_n / \lambda$ 
takes its values in $\{ -2n-1, -2n+1, \ldots , 2n-1 \}$.  Iterating the previous constructions shows that adding $h_n$ to the GFF with Dirichlet boundary and Neumann boundary values 
in the complement of these $n$ SLE$_4 (-1)$ layers constructs indeed a mixed-GFF $\Gamma^n$ in $\HH$ 
with boundary conditions given by $-\lambda$ on $\R_-$ and Neumann on $\R_+$. 

Let us now again fix a smooth function $f$ with compact support away from the real line. Then, it is easy to see that almost surely, there will be a value $n_0$ such that 
the $n_0$-th SLE$_4 (-1)$ layer will be under the support of $f$, so that for all $n \ge n_0$, $\Gamma^n (f) = \Gamma^{n_0} (f)$.
Hence (given that if a sequence of Gaussian random variables converges almost surely, it converges also in law to a Gaussian random variable),
we conclude readily that the limit $\Gamma^\infty (f)$ defines also mixed GFF in  $\HH$ 
with boundary conditions given by $-\lambda$ on $\R_-$ and Neumann on $\R_+$. 

In other words,  it is possible to first sample the entire infinite branching tree of SLE$_4 (-1)$ processes, together with their labels, and then 
to sample conditionally independent GFFs in each of the connected components of its complement (mind that all these GFFs have constant boundary conditions, and that the 
constant is in $\lambda + 2 \lambda \Z$), and that one gets exactly a GFF in $\HH$ with boundary values $-\lambda$ on $\R_-$ and Neumann on $\R_+$.

A third step is to notice that this SLE$_4 (-1)$ exploration tree is exactly the same as the one that one would have drawn if one would have explored the collection
of all boundary touching $0$-level lines of a GFF $\tilde \Gamma$ with boundary conditions $-\lambda$ on $\R_-$ and $0$ on $\R_+$. In particular, this shows that: 

\begin {proposition} 
\label {propoND}
 One can couple a GFF $\tilde \Gamma$ (with boundary conditions $-\lambda$ on $\R_-$ and $0$ on $\R_+$) with a GFF $\hat \Gamma$ 
 (with boundary conditions $-\lambda$ on $\R_-$ and Neumann on $\R_+$) in such a way that the only difference between $\tilde \Gamma$ and $\hat \Gamma$ lies in the signs of the  height jumps on the arcs of their common underlying branching tree. 
\end {proposition} 

Note that the previous ALE-type decomposition of $\hat \Gamma$ is in fact a deterministic function of $\hat \Gamma$; the traced arcs are exactly the 
collection of all level lines of $\hat \Gamma$ with height in $2 \lambda \Z$ that intersect $\R_+$.
%(this collection is by construction the same for $\hat \Gamma$ and for $\tilde \Gamma$, and for $\tilde \Gamma$, it is known that this collection is a deterministic function of $\tilde \Gamma$). 

\medbreak

Finally, to conclude the proof of Theorem \ref {mainthm}, 
we need to come back to the study of entirely Dirichlet vs. entirely Neumann boundary conditions.
Let us start with a Dirichlet GFF $\Gamma$, consider its boundary touching level arcs at height $0$, toss an independent fair coin for each of them, and define $\Lambda$ as in Theorem \ref {mainthm}. 
In order to prove that $\Lambda$ is a realization of a Neumann GFF, we use a limiting procedure:  
If we map the previous coupling of $\tilde \Gamma$ and $\hat \Gamma$ onto the unit disc, we see that, if we define $\partial_1^\epsil$ and $\partial_2^\epsil$ the two disjoint arcs on the unit circle separated by the points $\exp (i( \pi - \epsil))$  and $\exp (i (\pi + \epsil))$ for some very small $\epsil$, and one considers the GFFs $\hat \Gamma_\epsil$ and $\tilde \Gamma_\epsil$ with boundary conditions $-\lambda$ on the small arc $\partial_1^\epsil$  and on the long arc, $0$ boundary condition for $\tilde \Gamma_\epsil$ and Neumann condition for $\hat \Gamma_\epsil$, one can couple the 
two fields so that they share the same ALEs and the difference is described as above. 
Now, when one lets $\epsil \to 0$: 

- On the one hand, one can see that it is possible to couple all $\tilde \Gamma_\epsil$ with the GFF $\Gamma$ with Dirichlet boundary conditions, 
so that when one restricts it to any given set at positive distance from $-1$, then almost surely, $\tilde \Gamma_\epsil$ coincides exactly with $\Gamma$ for all small enough $\epsil$ 
(one way to see this is to use the coupling of $\Gamma_\epsil$ with a Brownian loop-soup + a Poisson point process of excursions away from the small arc, as for instance in \cite {Lupu2,QianW}). This shows in particular that all the ALE arcs of $\tilde \Gamma_\epsil$ that are at some positive distance from $-1$ will stabilize to those of $\Gamma$: More precisely, each level arc at height $0$ of $\Gamma$ will be almost surely a level arc of $\tilde \Gamma_\epsil$ for all sufficiently small $\epsil$. 

- If we couple each $\tilde \Gamma_\epsil$ with 
an $\hat \Gamma_\epsil$ as in Proposition \ref {propoND} and in such a way that the signs of the height-jumps on the boundary-touching level arcs of the latter coincide with those of $\Lambda$ on those arcs that are also boundary-touching level arcs of $\Gamma$, then the gradient of $\hat \Gamma_\epsil$ will stabilize as well. In fact, if we shift $\hat \Gamma_\epsil$ in such a way that the boundary conditions of the ALE cell that contains the origin are the same for $\hat \Gamma_\epsil$ as the value for $\tilde \Gamma_\epsil$ (this boundary value is  $\eps \lambda$ for $\eps = +1$ or $\eps = -1$), then this shifted field $\Lambda_\epsil$ will also almost surely stabilize to $\Lambda$.

- On the other hand, the law of $\hat \Gamma^\epsil$, when viewed as acting on functions with zero mean will converge to that of a Neumann GFF in the unit disk (this follows from the fact that 
the corresponding covariances converge), so that the limiting field $\Lambda$ is indeed a realization of the Neumann GFF. 

This concludes the proof of Theorem \ref {mainthm}. 

\medbreak
Let us summarize in a picturesque way the CLE-type description of the GFF with mixed  boundary conditions, in Figures \ref {pic9} and \ref{pic10}.  

 \begin{figure}[ht!]
\includegraphics[scale=0.55]{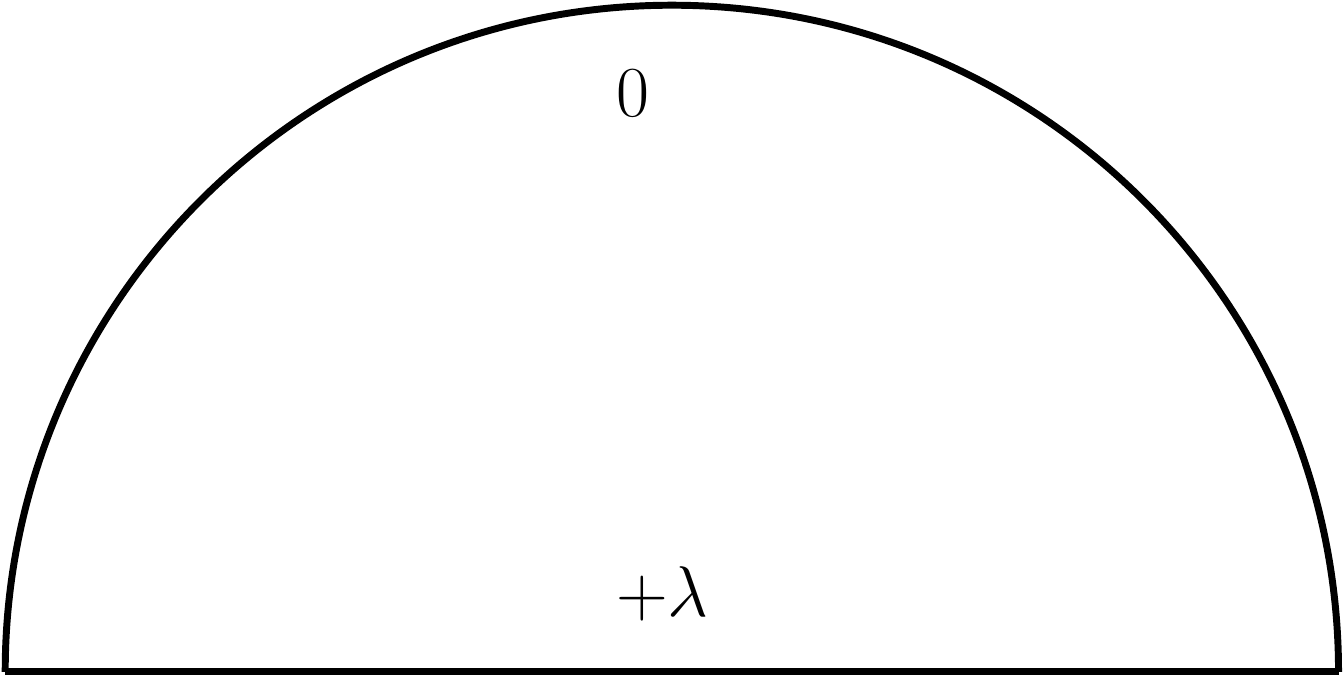}
\quad 
\includegraphics[scale=0.55]{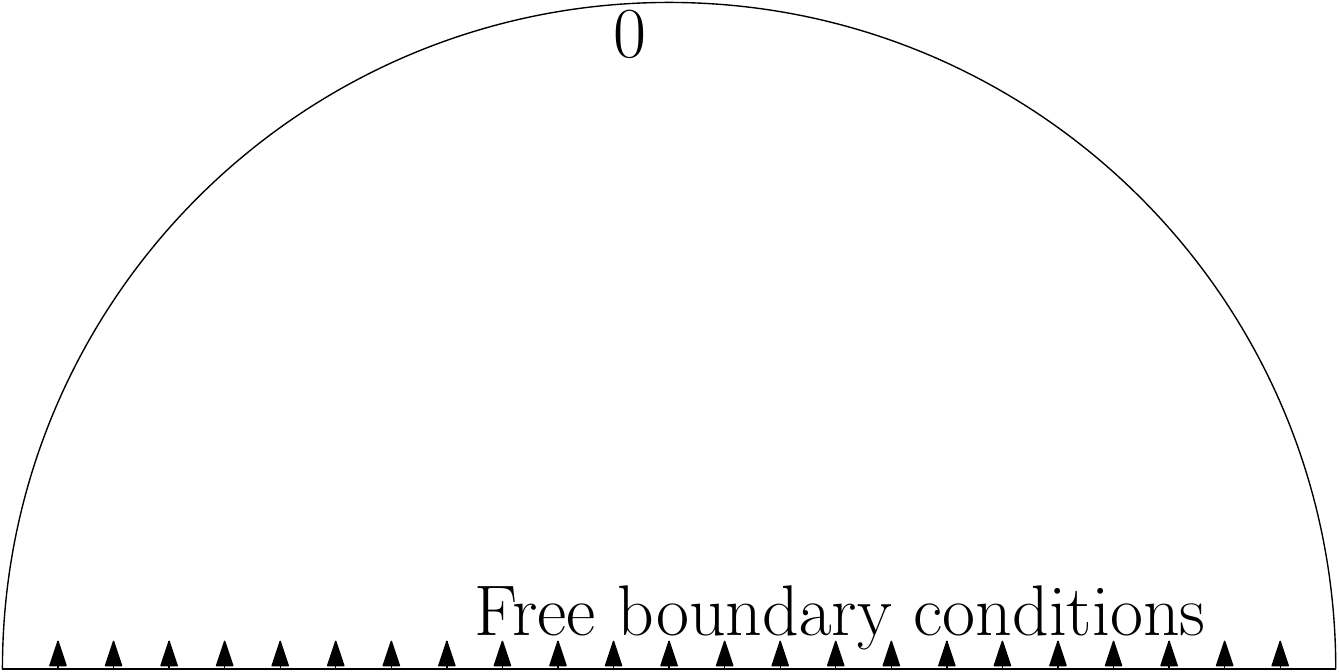}
\vskip 5mm
\includegraphics[scale=0.55]{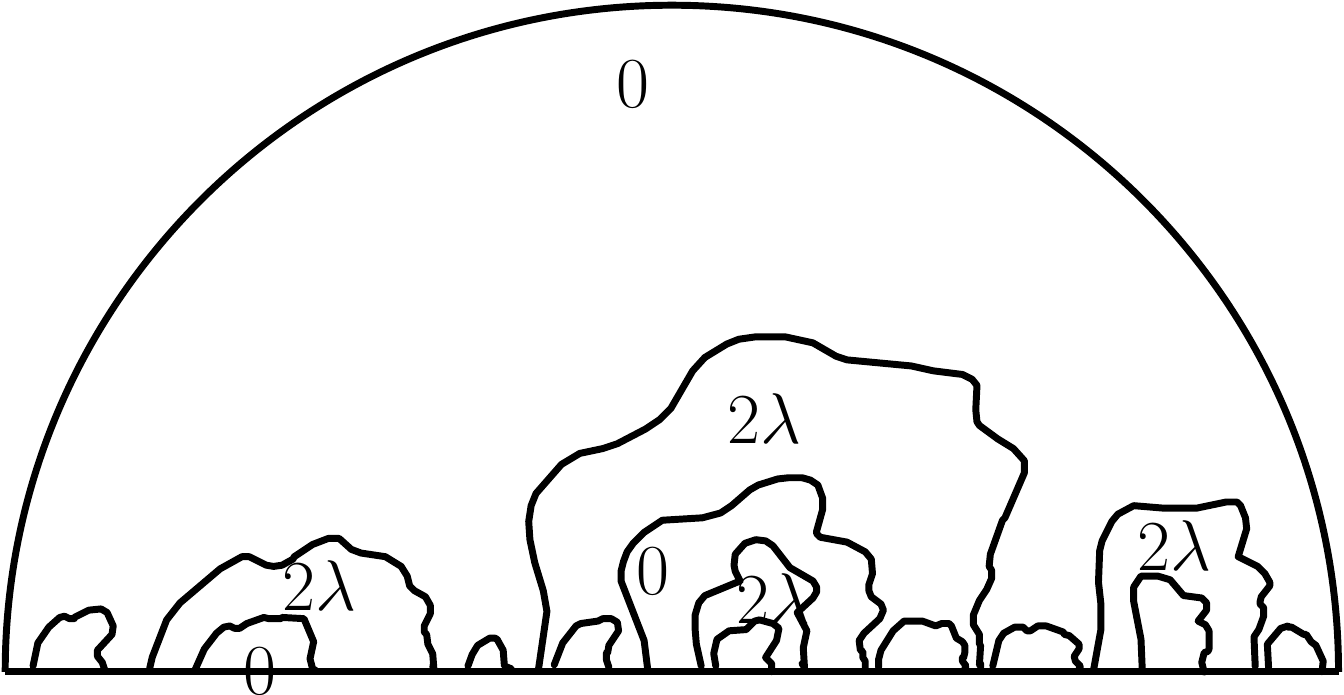}
\quad
\includegraphics[scale=0.55]{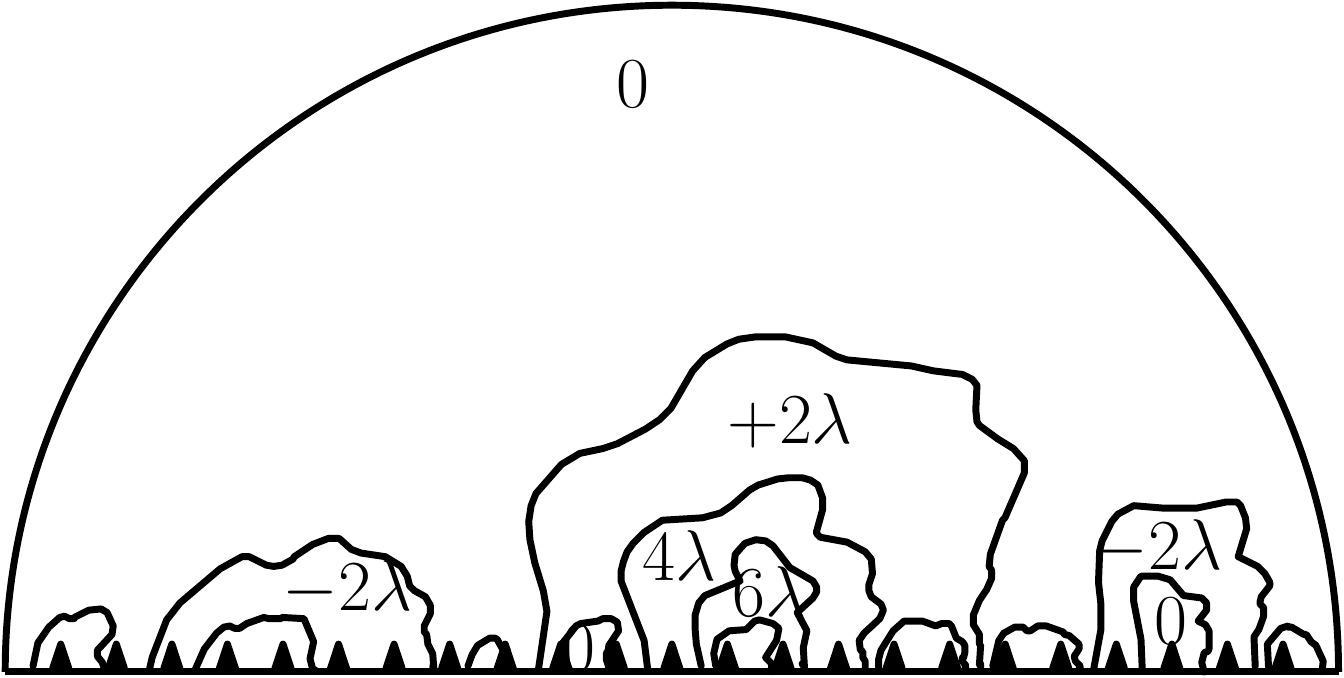}
  \caption {Top: The boundary conditions for $\tilde \Gamma + \lambda $ and $\hat \Gamma + \lambda $ in the semi-disc.
  Bottom: The same ALE with different heights. On the left, alternating $0$ and $2\lambda$ will create the $+\lambda$ boundary condition on the horizontal line. On the right, the random height-jumps will create the Neumann boundary condition on the horizontal line.}
  \label{pic9}\end{figure}

Let us make here one comment that will be useful later in the paper.  Consider the Neumann-Dirichlet ALE in the semi-disc as in Figure \ref {pic9}, with Neumann boundary conditions on the horizontal segment. 
Just as for CLE$_4$ in the unit disc, one can immediately relate the covariance of the field at two points $x$ and $y$ to the expectation of number $N(x,y)$
of loops or arcs that disconnect both these points from the semi-circle (indeed, the field in the neighborhoods of $x$ and $y$ will be $2 \lambda$ times the sums of the same $N(x,y)$ coin tosses 
plus zero-expectation terms that are conditionally independent for $x$ and $y$). More precisely (we state this as a lemma for future reference): 
\begin {lemma}
\label {arcs}
The Neumann-Dirichlet Green's function $G^{\mathcal {ND}}(x,y)$ in the semi-disk  (which is equal to the sum $G_{\U} (x,y) + G_{\U}
 (x, \overline y)$ where $G_{\U}$ is the Dirichlet Green's function in the unit disc) is equal to $(2 \lambda)^2$ times the expectation of $N(x,y)$. 
\end {lemma}
In particular, when $x$ and $y$ are  both on the segment $[-1, 1]$, then the Dirichlet-Neumann Green's function $G (x,y)$ in $D$ (mind that when $x$ and $y$ are on the real axis 
$G(y,x) \sim \pi^{-1} \log (1 / |y-x|)$ as $y \to x$, i.e., it explodes twice faster as when $x$ is not on the real axis) is
exactly equal to $(2\lambda)^2 = \pi /2$ times the expected number of nested ALE arcs that disconnects them both from the semi-circle in Figure \ref {pic9} 
(this number of arcs is then the same for both fields $\tilde \Gamma$ and $\hat \Gamma$).

\begin{figure}[ht!]
\includegraphics[scale=0.5]{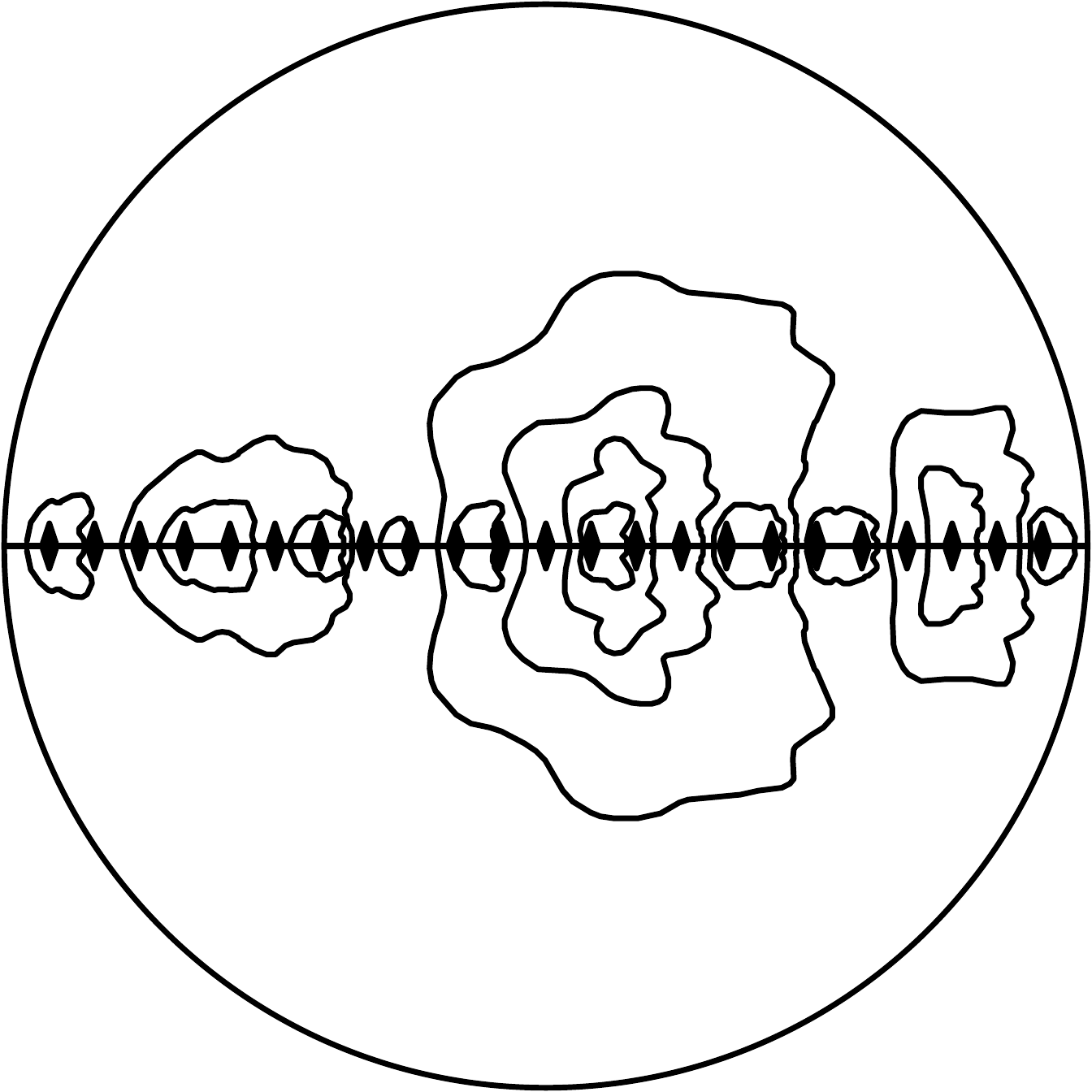}
\quad
\includegraphics[scale=0.5]{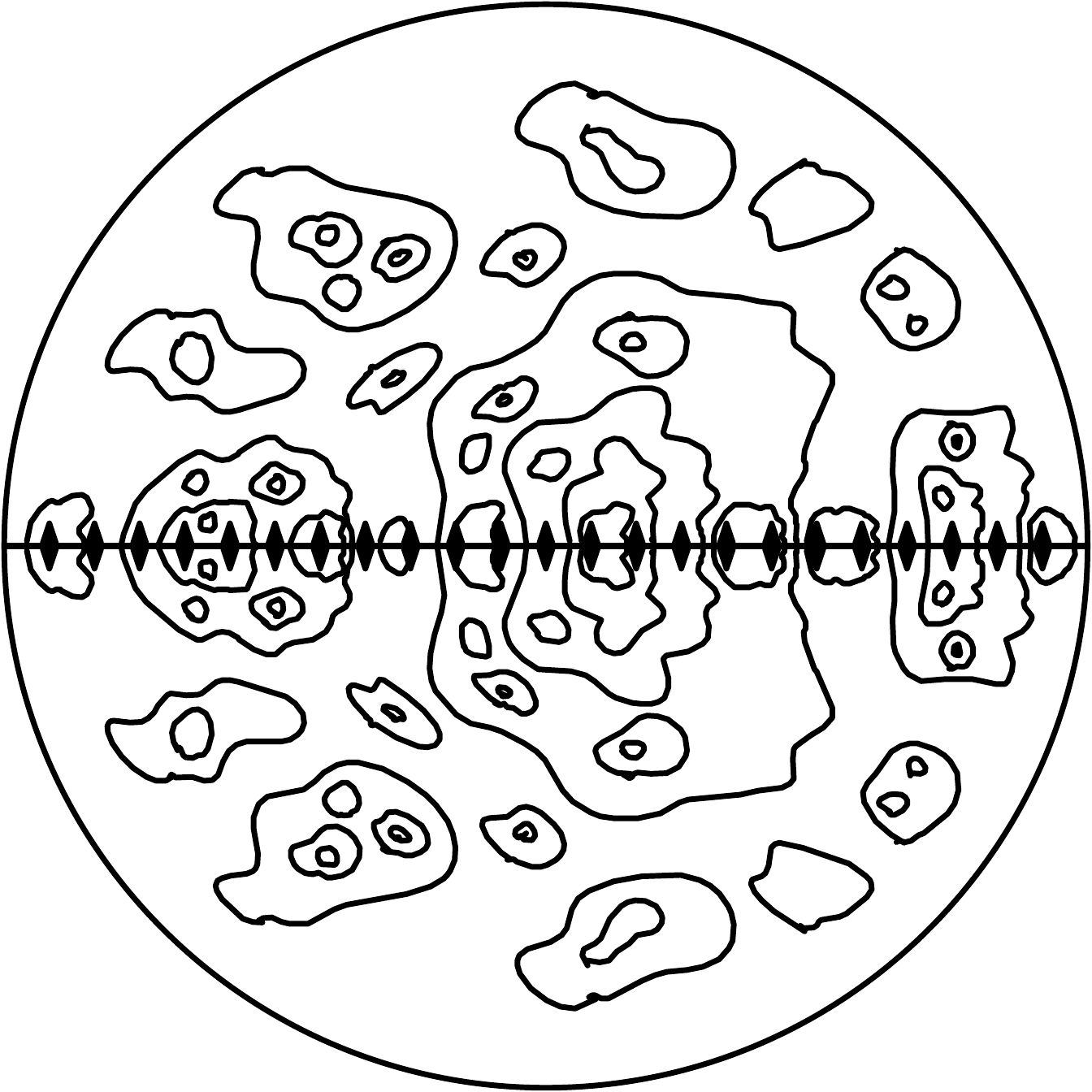}
  \caption {Left: Symmetrizing the boundary touching ALE for mixed Dirichlet-Neumann conditions creates a symmetric CLE$_4$-type picture. 
  Right: One can then complete with independent usual CLE$_4$ in the connected components of the complement in the top half of the picture, 
   toss an i.i.d.coin for the sign of the height-gap for each of the loops of the obtained figure, and one reconstructs a GFF with mixed Neumann-Dirichlet boundary conditions in the top semi-disc. If one forces the height-gaps of the 
   loops that intersect the horizontal axis to alternate, then one constructs the Dirichlet GFF.}
  \label{pic10}\end{figure}

\section {Relation to soups of reflected loops and consequences}
\label {S5}

In the previous sections, we have generalized  the construction of ALEs and CLE$_4$ that had been performed for the Dirichlet GFF to the cases of the Neumann GFF and of the GFF with mixed 
Neumann-Dirichlet boundary conditions.
The goal of the present section is to explain how to generalize the results of \cite {ShW,Lupu,Lupu2,QianW} that relate CLE$_4$ and the Dirichlet GFF to Brownian loop-soups, to the case of 
Neumann boundary conditions. 
Here, we assume that the reader has these papers (and in particular \cite {QianW}) in mind or at hand, as we will refer to them frequently. 

It is again more convenient to first focus on the case of mixed Dirichlet-Neumann boundary conditions -- the case of purely Neumann boundary conditions can then 
be deduced via a similar limiting procedure as above. Let us for instance choose to work in the semi-disc $D = \{z \ :  \ |z| < 1 , \Im (z) > 0 \}$ 
with Dirichlet boundary conditions on the semi-circle (that we will denote by $\partial$) and Neumann boundary conditions on the segment $I= [-1,1]$. We can then consider the Brownian motion in $D$, that is killed on $\partial$ and reflected orthogonally on $I$. This is a reversible Markov process, and all the construction of the corresponding loop measure and loop-soups 
works without further ado: 

- One can define the measure $\mu$ on (unrooted) loops of reflected Brownian motions in $D$ (i.e., reflected on $I$ and killed on $\partial$).
One can relate it directly to the usual measure $\mu_\U$ on unrooted Brownian loops in the unit disc as 
follows: Define $\varphi$ the ``folding'' map that is equal to the identity in $\overline D$ and to $z \mapsto \overline z$ on the symmetric image of $D$ with respect to the real axis. 
Then, the image of $\mu_\U$ under $\varphi$ is exactly $2 \mu$ (or equivalently, $\mu$ is the image of $\mu_\U / 2$ under $\varphi$). 

- One can define for each positive $c$, the soup ${\mathcal L}_c$ of reflected loops in $D$ (with reflection on $I$ and killing on $\partial$) as the Poisson point process of loops  with intensity
$c \mu$ (here we choose the normalization of $\mu_\U$ and therefore also $\mu$ such that the occupation time of the loop soup ${\mathcal L}^\U_c$  with intensity $c\mu_\U$ is related directly to the square of the GFF in $\U$ when $c=1$). We see that one way to construct ${\mathcal L}_c$ is to view it as the image under $\varphi$ of the loop-soup ${\mathcal L}^\U_{c/2}$.

- The appropriately normalized occupation time measure of the loop-soup ${\mathcal L}_1$ (that is, for $c = 1$)
is distributed exactly as the appropriately defined square of the GFF with mixed Dirichlet-Neumann boundary conditions in $D$. In summary, {\em the coupling between the loop-soup ${\mathcal L}_1$ and the square of the Neumann-Dirichlet GFF in $D$ works directly.} 

\medbreak

Let us now explain why the loop-soup clusters will give rise to the structures that we studied in the previous section, in the
same way in which loop-soup clusters in $\U$ give rise to CLE$_4$ as shown in \cite {ShW}. Let us first recapitulate the notation that we are using for the different loop-soups. For each $c$, 
${\mathcal L}_c$ will denote the Neumann-Dirichlet loop-soup in the semi-disc $D$ with intensity $c$, 
${\mathcal L}_c^\U$ and ${\mathcal L}_c^D$ will denote the usual Dirichlet loop-soups with intensity $c$ in $\U$ and $D$ respectively, and ${\mathcal L}_c^I$ will denote the set of 
loops of ${\mathcal L}_c$ that intersect and are reflected on $I$. Hence, for instance, ${\mathcal L}_c \setminus {\mathcal L}_c^I$ is distributed like ${\mathcal L}_c^D$, and the image 
of ${\mathcal L}_{c/2}^\U$ under $\varphi$ is distributed like ${\mathcal L}_c$. 

Not surprisingly, the collections ${\mathcal L}_c^I$ will satisfy a chordal conformal restriction property (see \cite {LSWr,Wcrrq} for background on those): 
\begin {lemma}
Consider the upper boundary of the union of all loops in ${\mathcal L}_c^I$. This is a simple curve that satisfies the chordal one-sided conformal restriction property in $D$ with marked point at $-1$ and $1$, with exponent $c / 16$. 
\end {lemma}

\begin {proof}
Note first that the fact that the upper boundary of the union of all loops in ${\mathcal L}_c^I$ satisfies the one-sided conformal restriction property for some positive exponent $ \alpha$ that is a constant
multiple of $c$ follows immediately from the conformal invariance of the reflected loop measure, and the definition of ${\mathcal L}_c^I$. It therefore only remains to identify the 
multiplicative constant $1/16$. 
The fact that the image under $\varphi$ of $\mu_\U$ is $2 \mu$ makes it possible to compute directly the mass for $\mu$ of certain families of loops. 

It will be convenient to work in the upper half plane. Consider the loop-measure for Brownian motion reflected on $\R_-$ and killed on $\R_+$, 
and let us estimate the mass $m(\eps)$ of the set of loops that touch both $\mathbb{R}_-$ and the very small vertical segment $[1,1+\eps i]$. By applying the $z \mapsto \sqrt{z}$ map and by considering the reflection of the first quarter plane along its vertical boundary, we get that it will be sufficient to estimate the mass of the (non-reflected) Brownian loops in $\mathbb{H}$ that intersect both the vertical half-line $i\R_+$ and the union of the two segments $[1,1+i\eps/2]$ and $[-1,-1+i\eps/2]$. When $\eps$ is small, this is approximately twice the mass of the loops that intersect $i\R$ and $[1,1+i\eps/2]$ (the mass 
of the set of loops that intersect both segments being a smaller order term). 
But this mass of the loops in the upper half-plane that intersect both $i\R_+$ and $[1,1+i\eps/2]$ is approximately
hcap$([1,1+i\eps/2])/2=\eps^2/16$ times the mass of Brownian bubbles in $\mathbb{H}$ rooted at $1$ and intersecting $i\R$.
By \cite{LW}, the mass of such bubbles is given by $-1/6$ times the Schwarzian derivative of the map $z\mapsto z^2$ at the point $1$, i.e., to $-Sf (1) / 6$ for $f(z) = z^2$. 
We can then conclude that $m(\eps)=\eps^2/32+o(\eps^2)$.

In order to relate this to the restriction exponent, we can recall that 
the probability that no loop in the Poisson point process $L_c$ of Brownian loops in $\HH$ (reflected on $\R_-$ and killed on $\R_+$) 
with intensity $c$  intersect both $\R_-$ and $[1,1+\eps i]$ is given by $\exp(-cm(\eps))$ which behaves therefore like $1-c\eps^2/32+o(\eps^2)$ in the $\eps \to 0$ limit. 
On the other hand, if $g_\eps$ denotes the conformal map from $\mathbb{H}\setminus [1,1+\eps i]$ onto $\HH$ that leaves the points $0,\infty$ invariant and such that $g_\eps(z)\sim z$ as $z\to \infty$, then $g'_\eps(0)=1-\eps^2/2+o(\eps^2)$. The probability that a one-sided restriction measure on $\R_-$ in $\mathbb{H}$ of exponent $\alpha$ does not intersect  $[1,1+\eps i]$ is given by $g'_\eps(0)^\alpha=1-(\alpha/2)\eps^2+o(\eps^2)$. Therefore we conclude that indeed $\alpha=c/16$.
\end {proof}

We can already note that this value $1/16$ is not surprising given our previous results relating the GFF with Neumann boundary conditions to SLE$_4 (-1)$, and the fact that 
  the upper boundary of the union of an independent restriction sample of exponent $1/16$ in $D$ attached to $I$ with all the loop-soup clusters in ${\mathcal L}^D_1$ that it intersects 
 is an SLE$_4 (-1)$, see \cite {WW2}. Hence, {\em the first SLE$_4 (-1)$ layer of the ALE-type construction is obtained by boundary touching clusters of Brownian loops of ${\mathcal L}_1$}. 
 This is a first step towards relating the previous level lines of the Neumann GFF constructions to Brownian loop-soup construction of the SLE$_4 (-1)$ and of the square of the GFF -- which will 
 be explained at the end of the present section.

 \begin{figure}[ht!]
\includegraphics[scale=0.5]{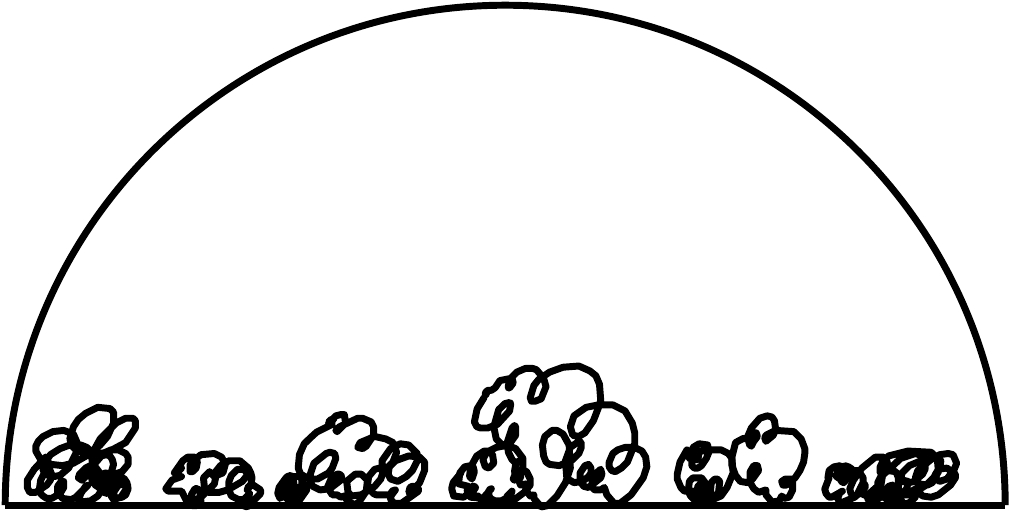}
\quad 
\includegraphics[scale=0.5]{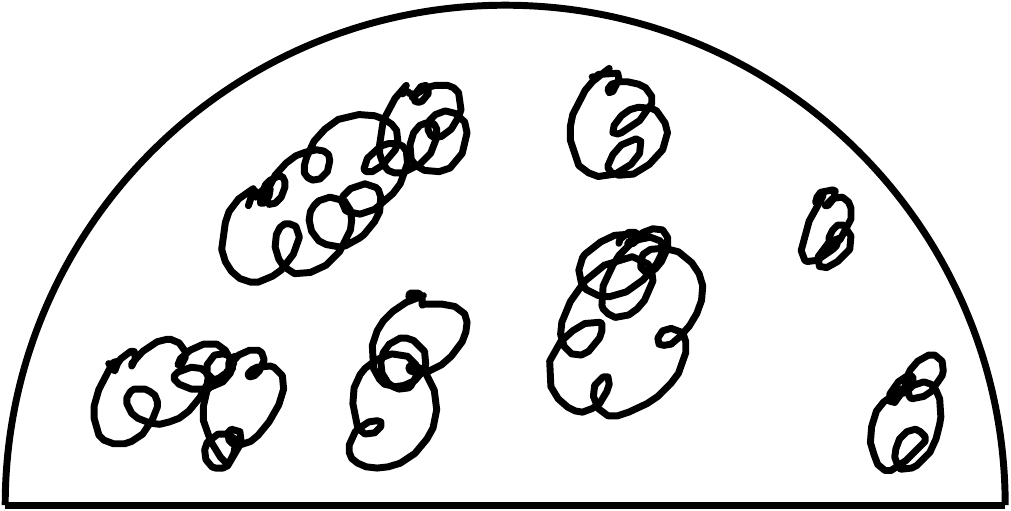}
\quad
\includegraphics[scale=0.5]{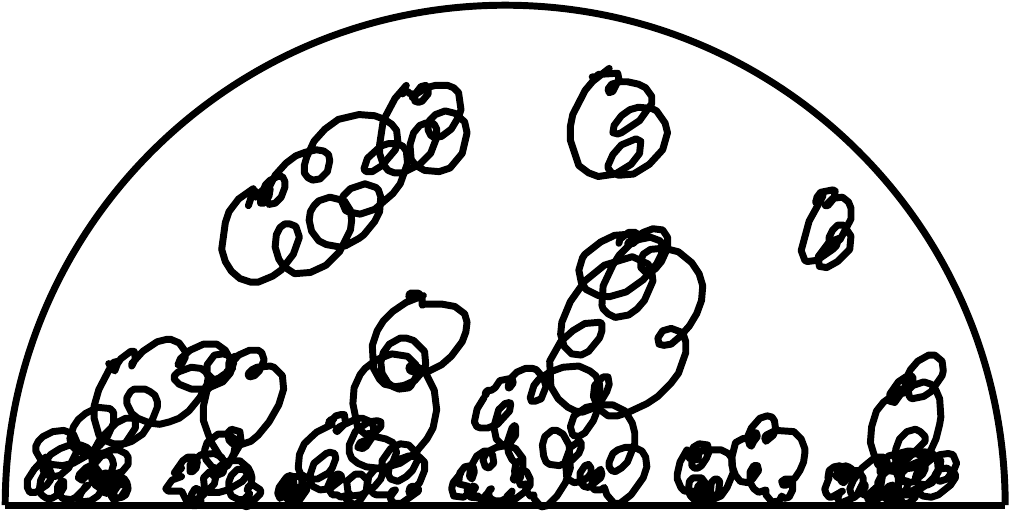}
  \caption {Sketch of ${\mathcal L}_c^I$ (left), ${\mathcal L}_c^D$ (middle) and ${\mathcal L}_c$ (right). The upper boundary of the union of the loops on the left picture is a restriction sample with exponent $c/16$.
  The upper boundary of union of the cluster touching $I$ in the right picture is an SLE$_{\kappa(c)} (\rho(c))$ -- when $c=1$, it is an SLE$_4 (-1)$. }
  \label{pic11}\end{figure}

 Let us now describe some consequences of the identification of this exponent $c/16$ for general $c \le 1$: 
 \begin {proposition}
 \label{previousproposition}
  Consider the loop-soup ${\mathcal L}_c$ with mixed boundary conditions and intensity $c$. Then, the upper boundary of the union of the loop-soup clusters that do touch $I$ (as depicted in Figure \ref {pic11}) is 
 an SLE$_\kappa (\rho)$ curve where $\kappa \in (8/3, 4]$ is related to $c$ by the usual identity $c = (6 - \kappa)( 3 \kappa - 8) / (2 \kappa)$ and 
 $$ \rho = \rho(c) = \sqrt {\frac {(8- \kappa) (\kappa -2) } 8 } - \frac { 8-\kappa } 2  .$$
 \end {proposition}
 
 \begin {proof}
We know that the upper boundary of the union of all loops in ${\mathcal L}_c^I$ forms a one-sided restriction sample with exponent $c/16$, and 
that ${\mathcal L}^D_c$ forms an independent loop-soup in $D$ (and therefore defines a CLE$_\kappa$ for $\kappa (c)$ as in the statement of the proposition). 
The proposition now follows immediately from the following result from  
 \cite {Wcras,WW2}: The upper boundary of the union of a restriction sample of exponent $\beta$ 
 with the   loops of an independent CLE$_\kappa$ that they do intersect is an SLE$_\kappa (\rho)$ curve from $-1$ to $1$ in $D$, 
 where $\rho$ is the unique value in $(-2, \infty)$ such that
$$ \frac {(\rho +2 ) ( \rho + 6 - \kappa ) }{4 \kappa } = \beta. $$ 
\end {proof}
When $c=1$, then $\kappa =4$ and one indeed gets $\rho (1) = -1$. But for instance when $c=1/2$, then $\kappa =3$ and 
$\rho (1/2) = \sqrt {5/8} - 5/2$ which is somewhat unusual expression in the SLE/GFF framework.  

Note that using Theorem 1.6 in \cite {MWu}, one gets the explicit expression of the dimension of the intersection of the ``Dirichlet-Neumann CLE$_\kappa$ carpet'' with this real line. 

\medbreak

This proposition describes the reflected loop-soup construction of the first layer of the ALE of the Neumann-Dirichlet GFF when $c=1$ and of its generalization for $c <1$. 
In order to construct the subsequent lower layers of the ALE, we can notice that the law of the picture that 
one observes when just looking at the trace of the loop-soup ${\mathcal L}_c$ in $D \cap \{ z \ :  \ \Im (z) > \eps \}$ is absolutely continuous 
with respect to the trace of the loop-soup ${\mathcal L}_c^D$ in $D \cap \{ z , \ \Im (z) > \eps \}$ (this can be seen by first noting that  almost surely, the number of 
loops of ${\mathcal L}_c$ that touch both $I$ and $\{z \ :  \  \Im (z) = \eps \}$ is finite, and that one could replace the portions of those loops below $\{ z \ :  \ \Im (z) < \eps /2 \}$ by portions of 
loops that do not touch $I$). 
This absolute continuity makes it possible to deduce from the corresponding result \cite {QianW} for loop-soups in $\U$, the following fact: 
{\sl Conditionally on the upper boundary $\delta$ of the closure of the union of the loop-soup clusters in ${\mathcal L}_c$ that intersect $I$, the conditional distribution of the Brownian loops of ${\mathcal L}_c$ that lie below $\delta$ and do not intersect $\delta$ is exactly a Brownian loop-soup in the domain between $\delta$ and $I$, with intensity $c$ and Dirichlet boundary conditions on $\delta$ and Neumann boundary conditions on $I$.} This is due to the fact that resampling the very small loops that lie near $\delta$ is already sufficient to ensure that the union of those loops with the Brownian loops that do intersect $\delta$ form connected clusters, as explained in \cite {QianW}.

This makes it possible, inside each of the connected components of $D \setminus \delta$ that are squeezed between $\delta$ and $I$
to use the loops of the original loop-soup ${\mathcal L}_c$ in order to construct the next layers of loops and so on. 
In summary: {\sl The construction of nested CLE$_\kappa$ using one single loop-soup that we pointed out in \cite {QianW} can be generalized in order to construct the entire collection of nested 
ALE arcs (i.e., their SLE$_\kappa (\rho)$ generalization for $c < 1$) using one single loop-soup ${\mathcal L}_c$.}

\medbreak

Further results worth pointing out are the following direct consequence of the fact that $\varphi ({\mathcal L}_{c/2}^\U )$ is distributed as ${\mathcal L}_c$. In the following statement, 
we consider $c \le 1$, and we define $\tilde \kappa$ and $\kappa$ in $(8/3, 3]$ resp. in $(8/3, 4]$ to be the values respectively associated to $c/2$ and $c$:  
$$ (6-\kappa)(3 \kappa - 8)/ (2 \kappa) = c \hbox { and } (6 -\tilde \kappa) (3\tilde \kappa - 8) / (2 \tilde \kappa) =  c/2.$$

\begin {corollary} Consider a standard CLE$_{\tilde \kappa}$ in the unit disc (viewed as a random collection of disjoint loops in the unit disc) 
 and consider its image under $\varphi$. In this picture, one now has collections of possibly overlapping simple loops in $D \cup I$, that can be decomposed into connected components. Then, the upper boundary of the union of all connected components that touch $I$ is distributed exactly like the SLE$_{\kappa} (\rho)$ path described in Proposition \ref {previousproposition}. 
 \end {corollary}
 
Here, the case $c=1$ turns out to be special as it allows to relate directly the usual CLE$_3$ (which is the scaling limit of the Ising model (see \cite {BH}) 
directly to the Neumann-Dirichlet GFF (note that just as the previous corollary, the following statement does not mention loop-soups but that its proof  very much uses them) -- we write this as a separate statement for future reference
(see Figure \ref {pic12} for an illustration of the result): 
\begin {corollary}
Consider a standard CLE$_3$ in the unit disc (viewed as a random collection of disjoint loops in the unit disc), and consider its image under $\varphi$. In this picture, one now has collections of possibly overlapping simple loops in $D \cup I$, that can be grouped into connected components. Then, the upper boundary of the union of all connected components that touch $I$ is distributed 
exactly like the SLE$_4 (-1)$ path that defines the first-layer of the ALE of the Dirichlet-Neumann GFF.
 \end {corollary}

 \begin{figure}[ht!]
\includegraphics[scale=0.55]{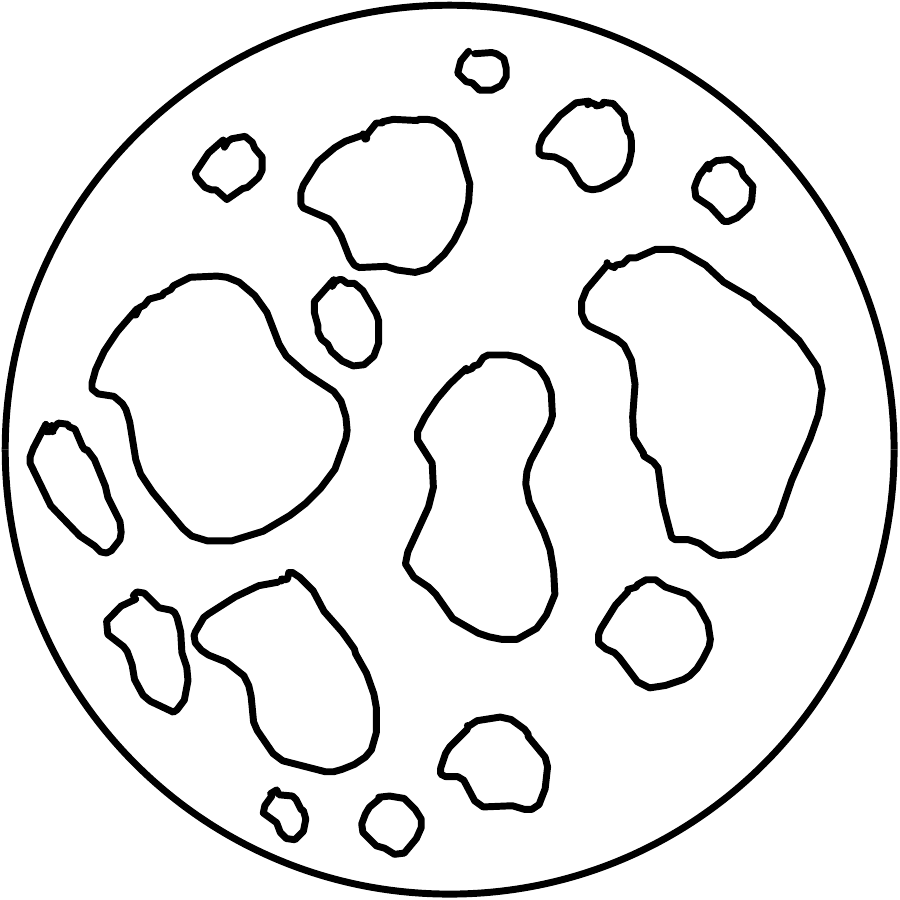}
\quad 
\includegraphics[scale=0.55]{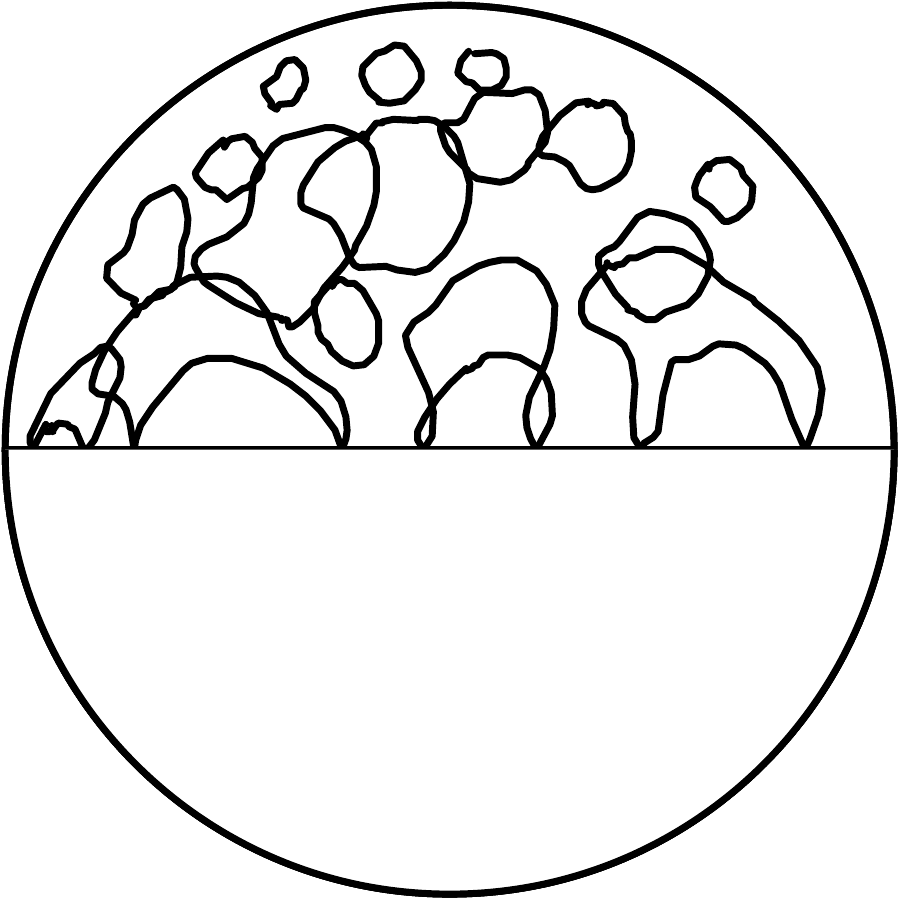}
\quad
\includegraphics[scale=0.55]{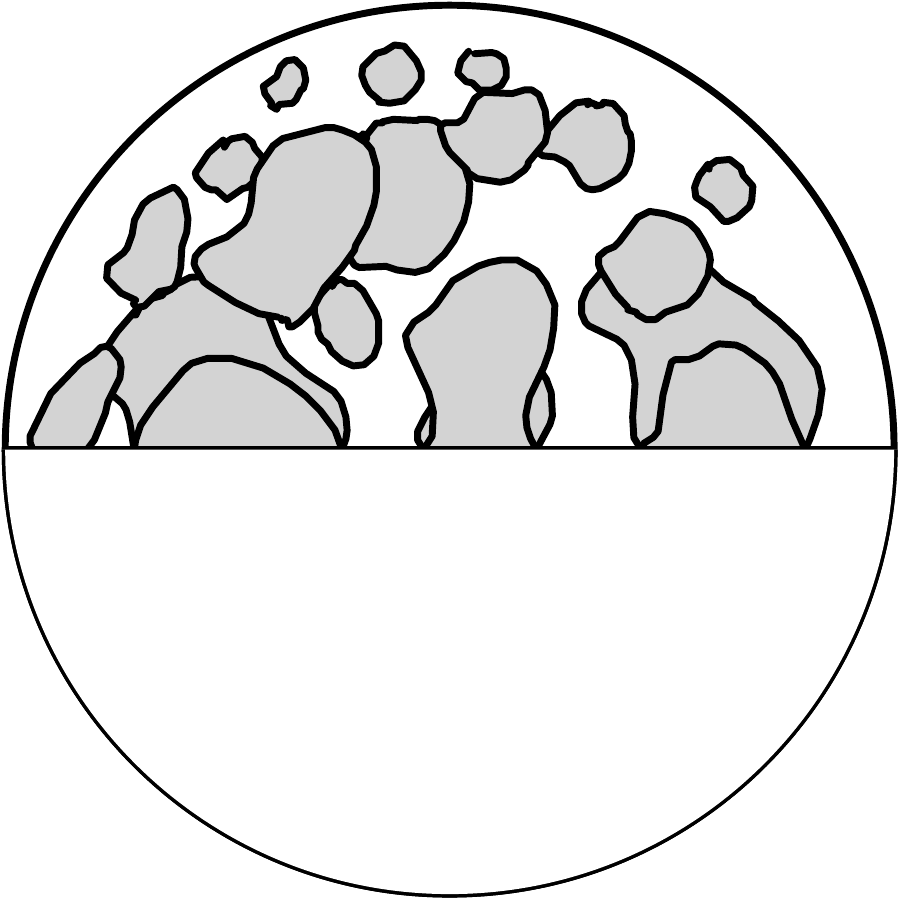}
  \caption {From CLE$_3$ to the Neumann-Dirichlet CLE$_4$ (sketch): A CLE$_3$ (left), the overlay of the CLE$_3$ and its symmetric image (middle), the obtained clusters form the Neumann-Dirichlet CLE$_4$ (right).}
  \label{pic12}\end{figure}

 \medbreak
 Note that one can also study in a similar way the soup ${\mathcal L}_c^D \cup {\mathcal L}_{c'}^I$ for all $c \in [0,1] $ and $c' \ge 0$. This 
 gives rise in the same way to nested SLE$_\kappa (\rho)$ layers 
 for $\kappa = \kappa (\rho)$ and all possible $\rho > -2$ (keeping $c$ fixed and letting $c'$ vary), More precisely, 
  the upper boundary of the union of the obtained clusters that do touch $I$ forms an SLE$_\kappa (\rho)$ 
process where $\kappa \in [8/3, 4] $ and $\rho > -2$ are determined by 
$$
c= \frac { (6-\kappa) (3 \kappa -8)}{2 \kappa} \hbox { and } \frac {c'}{16} = \frac { ( \rho +2) (\rho + 6 - \kappa) } {4 \kappa}. 
$$

Let us state the special case $c=1/2$ and $c'=1$ as a separate corollary: 
\begin {corollary}
Consider the overlay of ${\mathcal L}_{1/2}^D$ with an independent ${\mathcal L}_{1}^I$. 
Or equivalently, consider a loop-soup ${\mathcal L}_{1/2}^\U$ in the unit disc, remove all loops 
that are strictly contained in the lower half-plane, and take the image under $\varphi$ of the remaining ones. Then,
the upper boundary of the union of the obtained clusters that do touch $I$ forms an SLE$_3 ( -3/2)$ process.
\end {corollary}

It is interesting to emphasize that the various players in this result can be related directly to the Ising model \cite {BH,HK,cdhks}: 

- The CLE$_3$ in $D$ defined by ${\mathcal L}_{1/2}^D$ is the scaling limit of the boundary touching $+$ cluster for a critical Ising model with $+$ boundary conditions
(more precisely, the CLE$_3$ carpet which is the set surrounded by no loop has the law of this scaling limit). 

- The SLE$_3 ( - 3/2)$ process is
the scaling limit of
the lower boundary of the $+$ Ising cluster touching $\partial$ for the critical Ising model with $+$ boundary conditions on $\partial$ and 
free boundary conditions on $I$ (mind that the free boundary conditions for the Ising model are a priori not related to the Neumann boundary conditions for the GFF). 

Hence, this last corollary can be viewed as a rather simple coupling between the scaling limit of the critical Ising model with mixed boundary conditions ($+$ on $\partial$ and free on $I$) and the scaling limit of the Ising model with $+$ boundary conditions.

\medbreak
Finally, we come back to the $c=1$ loop-soup and to its relation with the GFF. We now have seen three couplings in this Dirichlet-Neumann setting: 

- Between the ALE structure and the GFF (this is the coupling described in the previous sections).

- Between the ALE structure and clusters of reflected loops (that we have just pointed out). 

- Between the soups of reflected loops and the square of the GFF (this is the isomorphism theorem \`a la Le Jan \cite {LJ} that works here without further ado).

We now briefly argue that, just as for the analogous case of the Dirichlet GFF \cite {QianW}, these three couplings can be made to coincide. Let us browse through the steps of the arguments: 
 The idea will be to realize the coupling of the GFF, of the ALE structure and of the loop-soup as the limit when the mesh-size goes to $0$ of a corresponding coupling built on metric graphs. In this discrete case, the 
key is the observation by Lupu that one can sample the GFF starting from its square (that is defined via the soup of Brownian loops on the metric graph) by choosing independent signs for different 
loop-soup clusters, because they correspond exactly to the excursion sets away from $0$ by the GFF on the metric graph. So,
in the discrete metric graph structure, one has a soup of reflected Brownian loops that create clusters and the square of the GFF
with the mixed Neumann-Dirichlet conditions, and one can the construct the Neumann-Dirichlet GFF itself { by taking the square root of 
the density of the occupation time measure, and sampling independently the signs of the GFF for each cluster.} 

When the mesh-size goes to $0$, we know that: 

- The discrete GFF with these mixed boundary conditions converges to some continuous GFF  
{ $\Lambda$ (the convergence in law can be seen from the convergence of the covariance functions because one is looking at Gaussian processes). }

{ - The discrete loop soup converges to a continuous one -- see \cite {LTF} for related convergence of loop-measure issues. }

-  The renormalized occupation time of the discrete loop-soup converges to { $T$, which is the renormalized occupation time of the continuous loop-soup.} 

{ - As explained in \cite{QianW}, we can make these convergences simultaneous such that $T$ is exactly the renormalized square of $\Lambda$.}

- The  clusters of macroscopic Brownian loops are exactly described via the SLE$_4 (-1)$ and CLE$_4$  structures that we described above (this is a statement about the Brownian loop-soup). On the other hand, they can only be larger than the limits of the clusters of the random walk loops when the mesh of the lattice goes to $0$.

Let us outline in a handwaving way one possible way to 
argue that the limit of the cluster of the cable system loops cannot be strictly larger than the corresponding cluster of the limiting macroscopic Brownian loops. 
First, we can notice that for both the version on cable systems as for the continuum version, when one 
restricts the loop-soup to the subset of $D$ that lies at distance at least $\eps$ from 
$I$, the picture is absolutely continuous with respect to the corresponding picture with Dirichlet boundary conditions on $I$ instead of Neumann boundary conditions. Furthermore, this Radon-Nikodym derivative (on the cable systems) does  
converge when the mesh of the lattice goes to $0$ (for each fixed $\eps$). 
If one combines this with the main result of Lupu \cite {Lupu2} (that for in the Dirichlet case, the scaling limit of the cable-system loop-soup clusters are never strictly larger than the clusters of macroscopic Brownian 
loops), we conclude that the only way in which in the Dirichlet-Neumann case, the limit of the cable-system loop-soup clusters could be strictly larger than the clusters of Dirichlet-Neumann Brownian loops would be that 
in the latter case, two boundary-touching clusters of loops would end up being connected ``through $I$''. In other words, the scaling limit of the cable-system loop-soup clusters would consist of 
 families of boundary touching clusters of Brownian loops that are only ``glued together'' along the boundary. 
 { In any case, the boundaries of the discrete clusters do converge to the boundaries of the continuous ones. This already enables us to apply the same arguments as in \cite{QianW}, and argue that the outermost  level lines of $\Lambda$ at height $\lambda$ and $-\lambda$ (which consists of 
 the SLE$_4(-1)$ and the CLE$_4$ loops above it) are exactly the boundaries of the loop soup clusters. This is because we can explore the loop soup from $\partial$ and all the arguments in \cite{QianW} remain valid as long as we do not discover any clusters touching $I$.

But on the other hand, conditionally on the SLE$_4(-1)$, $\Lambda$ restricted to each connected component under the SLE$_4(-1)$ curve 
have $\pm2\lambda$ boundary conditions according to i.i.d. fair coins. If some of the discrete clusters are indeed glued together along $I$ in the limit, then $\Lambda$ restricted to the corresponding components under the SLE$_4(-1)$ curve would take the same sign, instead of independent ones.
This leads to a contradiction. 
Hence the limit of the discrete clusters are exactly equal to the continuous ones.}

\begin{figure}[ht!]
\includegraphics[scale=0.6]{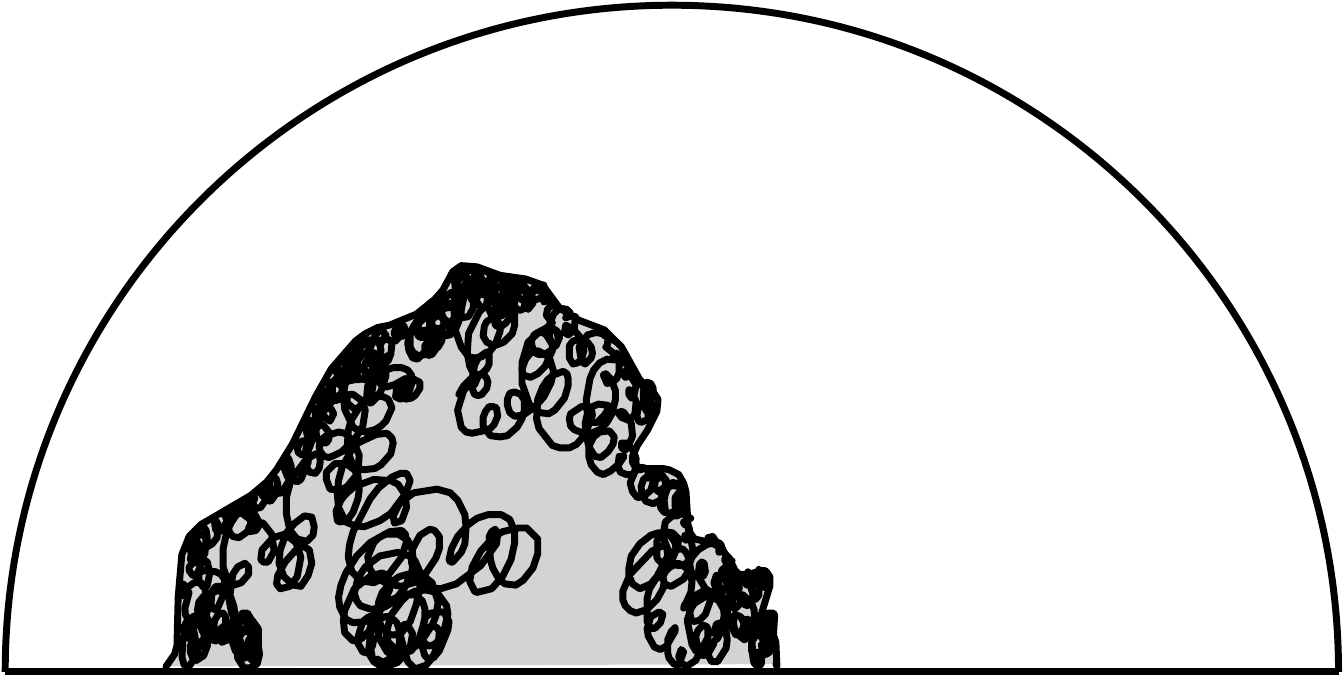}
\includegraphics[scale=0.6]{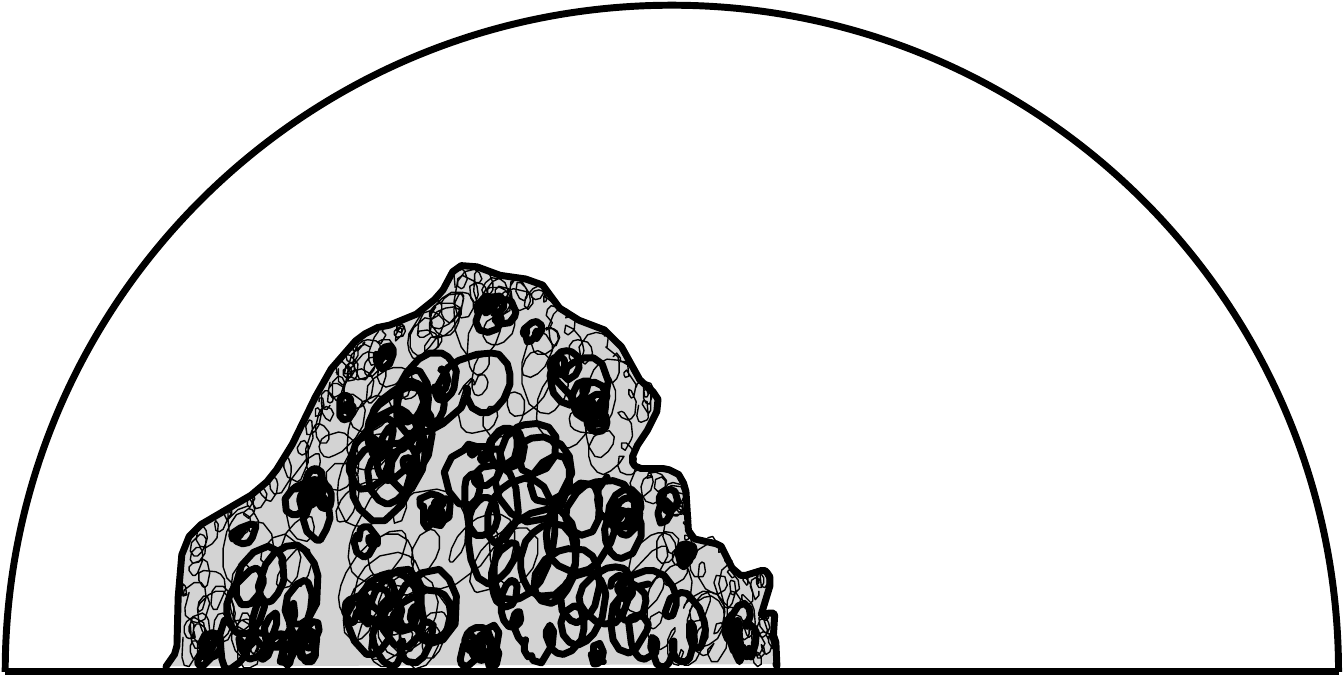}
  \caption {Illustration of the loop-soup cluster decomposition. Left: The upper boundary $\delta_j$  of a loop-soup cluster of reflected Brownian loops for $c=1$, with the loops in the Neumann-Dirichlet loop-soup that do touch this upper boundary. The union of these loops is then distributed like a Poisson point process of Brownian excursions away from this upper boundary and reflected on the real line. Right: The other Brownian loops under $\delta_j$ form a conditionally independent Neumann-Dirichlet soup of reflected loops in the domain inbetween $\delta_j$ (Dirichlet boundary) and the real line (Neumann boundary).}
  \label{decomposition}\end{figure}
\medbreak

With this in hand, it is then possible 
to adapt the arguments of the proof in the Dirichlet case that we presented in \cite {QianW}. 
This allows to generalize also the corresponding result of \cite {QianW} to this case of mixed Neumann-Dirichlet boundary conditions:
Conditionally on $\delta$, the trace of the set of Brownian loops in ${\mathcal L}$ that do intersect $\delta$ is distributed exactly like the union of 
a Poisson point process of Brownian excursions in $U(\delta)$, away from $\delta$ and reflected on $I$, with intensity $\alpha = 1/4$  (see Figure \ref {decomposition}). 
In other words, the relation between the previous three couplings and Dynkin's isomorphism work all in the same way as in the Dirichlet setting.

\section {Shifting the scalar version of the Neumann GFF and related comments}
\label {Sshift}

We now come back to the proof of Proposition \ref {propshift}. 
By symmetry, we can restrict ourselves to the case where $a \le 0$. 

Recall first very briefly how the coupling between the Neumann GFF, its realization $\Lambda$ and the corresponding ALE $A$ was established: 
We considered a GFF  $\hat \Gamma_u$ in the unit disc with mixed boundary conditions, free on a large arc and $-\lambda$ on the remaining very small arc $\partial_1^u$ around 
the boundary point $-1$. Then, we coupled it with a GFF $\tilde \Gamma_u$ with $-\lambda$ boundary conditions on the small arc and $0$ boundary conditions on the large arc, in such a way that the collection of boundary touching level arcs with height in $2 \lambda \Z$ of the two fields did exactly coincide. 
Then, we shifted $\hat \Gamma_u$ by a multiple of $2\lambda$ so that the height of the cell containing the origin was equal to that of $\tilde \Gamma_u$, and called this field $\Lambda_u$. Finally, we did let the size of the short arc shrink, and looked at the limit of $\Lambda_u$ and of its collection of boundary-touching level arcs with height in $2 \lambda \Z$, which provided the coupling of $\Lambda$ with the ALE. 

Let us define $b  = -\lambda + a \le -\lambda$,  and consider a GFF $\hat \Gamma^b_u$ in the unit disk, with mixed boundary conditions just as $\hat \Gamma_u$ except that the boundary condition on $\partial_1^u$ is $b$ 
instead of $-\lambda$. Note that this field is distributed exactly as $a + \hat \Gamma_u$. In particular, the boundary-touching level arcs of $\hat \Gamma^b_u$ with height in $2\lambda \Z$ will be distributed exactly like the boundary-touching level arcs of $\hat \Gamma_u$ with height in $-a + 2 \lambda \Z$. The argument at the end of the proof 
of the theorem showed that for a well-chosen coupling, the shifted fields $\hat \Lambda_u$ could be made to coincide away from $-1$ for sufficiently small $u$. This implies that the law of the collection of boundary-touching level arcs of $\hat \Gamma^b_u$ with height in $2 \lambda \Z$ does stabilize as well, so that in the coupling of the theorem, the 
law of the collection 
of boundary-touching level arcs of $\Lambda$ with height in $-a + 2 \lambda \Z$ is the limit as $u \to 0$ (for an appropriate topology) of the law of the collection of boundary-touching level arcs of $\hat \Gamma_u^b$ with height in $2 \lambda \Z$. 
In order to prove Proposition \ref {propshift}, we therefore just have to argue that this limiting law is that of an ALE, i.e., that it does not depend on the value of $b$. 
This will be a direct consequence of the following lemma: 
\begin {lemma}
\label {smalllemma}
For each $b< -\lambda$ and for a given $u_0$, there exists an arc $\hat \gamma$ that separates the small arc $\partial^{u_0}_1$ from $1$ in $\U$, such that given $\hat \gamma$,
the conditional law of $\hat \Gamma_{u_0}^b$ restricted 
to the connected component $\hat U$ of $\U \setminus \hat \gamma$ that has $1$ on its boundary, is exactly a mixed Neumann-Dirichlet GFF with Neumann boundary conditions of the unit circle part of the boundary of $\hat U$, and $- \lambda$ boundary conditions on $\hat \gamma$. 
\end {lemma}
This lemma readily implies Proposition \ref {propshift}. Indeed, we can choose to start with $\hat \Gamma_{u_0}^b$ and then to define $\hat \Gamma_u^b$ for all $u < u_0$ as the image of $\hat \Gamma_{u_0}^b$ via the M\"obius transformation of the unit disc that maps $1$ onto itself and $\partial_1^{u_0}$ onto $\partial_1^u$. 
Then, as $u$ tends to $0$, the diameter of the image of $\hat \gamma$ under this automorphism vanishes. In particular, the conformal transformation that maps $\hat U$ onto $\U$ that keeps $1$ fixed, has a derivative at $1$ that is equal to $1$ and maps $\hat \gamma$ onto some $\partial_v^1$ tends almost surely to the identity map as $u \to 0$ (for the uniform convergence when restricted to any subset of $\U$ that is at positive distance from $-1$). Hence, the collection of boundary-touching level lines with height in $2 \lambda \Z$ of $\hat \Gamma_u^b$ does indeed 
converge in distribution to an ALE.

Hence, it remains to explain how to derive Lemma \ref {smalllemma}: 
\begin {proof}
By conformal invariance, and shifting the field by $-\lambda$, we can consider a GFF 
 $\cech \Gamma$ in the upper half-plane $\mathbb {H}$ with mixed boundary conditions: Neumann on $\R \setminus [-1,1]$ and equal to $a  \le 0$ on $[-1,1]$. We will
 find an arc $\cech  \gamma$ joining a point in $(- \infty, -1)$ to a point in $(1, \infty)$ for this field, so that conditionally on this arc, the law of the field in the unbounded connected component of its complement is a mixed Neumann-Dirichlet GFF with Neumann conditions on the part of the boundary that is on the real line and $0$ boundary conditions on that arc.
 Such a coupling can in fact be obtained in a number of different ways.  A natural option (we will mention other options after the end of the proof) is to use the coupling of the Dirichlet-Neumann GFF $\cech \Gamma$ with a Dirichlet-Neumann loop-soup in $\HH$ and a Poisson point process of excursions away from $[-1,1]$ in $\HH$ that are reflected on $\R \setminus [-1,1]$ (see Figure \ref {fig:levela}).  This coupling can be viewed as the fine-mesh limit of the corresponding coupling on a cable-system approximation; on the cable system, in order to define the GFF, one considers the clusters of loops+excursions, the square of the GFF is then the intensity of the total occupation time measure, and tosses an independent fair coin to decide the sign of the GFF for each cluster, except those that contain an excursion away from the boundary segment, for which the sign is negative (this is exactly Dynkin's isomorphism theorem phrased in terms of cable-system loop-soups as in Lupu \cite {Lupu}, see also \cite {ALS}).
\begin{figure}[h!]
\centering
\includegraphics[width=0.48\textwidth]{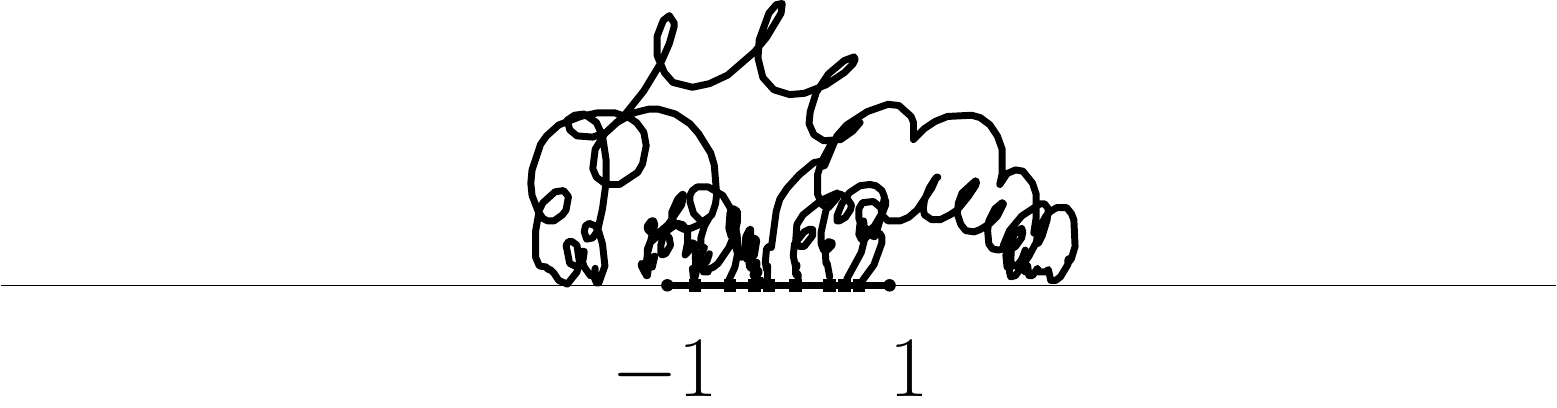}
\includegraphics[width=0.48\textwidth]{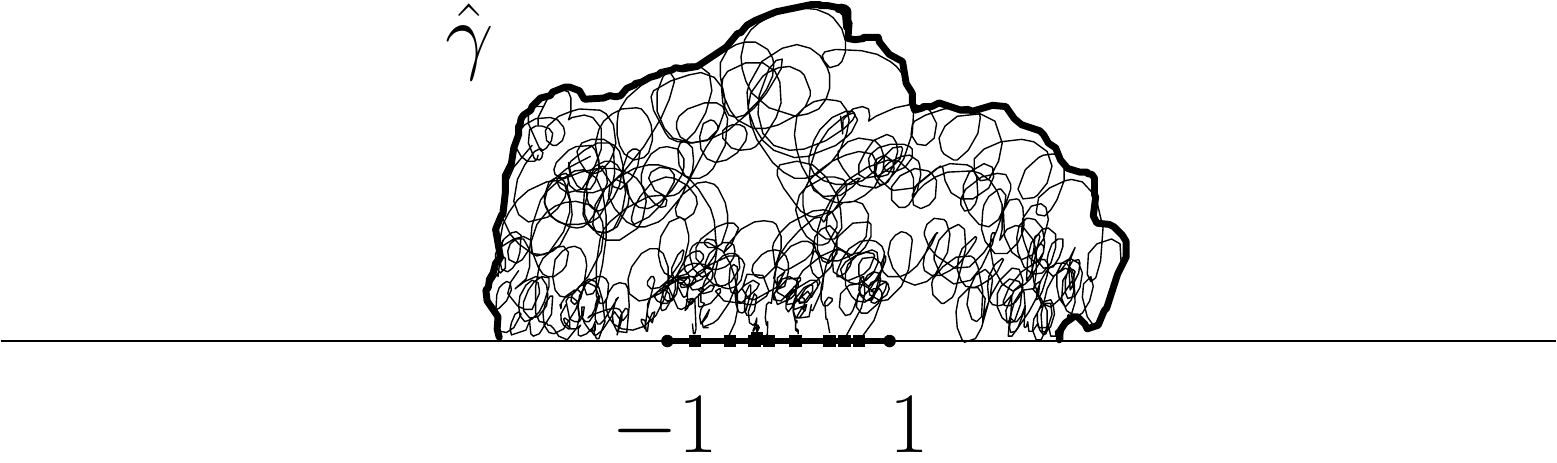}
\caption{The Poisson point process of excursions away from $[-1,1]$ reflected on the other part of $\R$ (left). Adding the loop-soup clusters of the Neumann-Dirichlet loop-soup that intersect them and drawing the outer boundary $\cech \gamma$ (right).}
\label{fig:levela}
\end{figure} 
 On the cable-system, one can look at the union of all connected components that do touch the segment $[-1, 1]$. Conditionally on this set, the law of the 
 cable-system GFF on the complement of this set is clearly a mixed Neumann-Dirichlet GFF, with Neumann boundary conditions on the real line and $0$ boundary conditions on the other boundary points. 
 When the mesh-size of the lattice goes to $0$, we can note that the outermost boundary of this set will tend to the corresponding outermost boundary of the union of all excursion+loop clusters that touch $[-1,1]$, which leads indeed the desired coupling between $\cech \Gamma$ and $\cech \gamma$ (a similar argument was used in the proof of Lemma 6 in our paper \cite {QianW}). 
 \end {proof}

 Let us make some comments about this proof: 
It is worthwhile emphasizing that for the boundary point at infinity, an apparent contradiction is that the expected value of the field $\cech \Gamma$ was $b + \lambda$ while it becomes $0$ for the mixed Neumann-Dirichlet (with zero values) above $\cech \gamma$. This is just due to the fact that the local set consisting of the union of all excursions and loop-soup clusters that they hit is not a thin local set (loosely speaking, it carries mass of the GFF) and that the probability that it is very large does not decay fast (see \cite {Se,ALS} for related considerations).  
 
 Let us note that instead of using the loop-soup approach, a second closely related option would have been to use absolute continuity between the laws of the field $\cech \Gamma$
 and the Neumann-Dirichlet GFF in $\HH$ with $-\lambda$ boundary conditions on $[-1,1]$ instead of $b$ boundary conditions. Indeed, when one restricts these two fields to the complement of a ball of large radius $R$ around the origin, then the Radon-Nykodim derivative between the law of the two tends (in probability) to $1$, so that one can then use directly the level-lines for $\hat \Gamma$ that we described in the previous sections. In some sense, the argument described in our proof was making use of an explicit way to couple these two fields. 
 
Another approach would for instance have been to adapt our proof of Lemma \ref {l3} to the field $\cech \Gamma^b$ for $b \in (-2\lambda, 0)$, i.e., to describe explicitely 
 a curve $\cech \gamma'$ via the concatenation of some SLE$_4 (\rho_1, \rho_2)$-type excursions, using the coupling between GFF with piecewise constant/Neumann boundary conditions and SLE$_4$ type curves before the curve hits the Neumann boundary as derived by Izyurov-Kyt\"ol\"a in \cite{IK}. Here, the boundary conditions on the complement of $\cech \gamma'$
 would be free on $\R \setminus [-1, 1]$, $b$ on the part that is in $[-1,1]$, and they would be exactly as in Lemma \ref {l3} on the curve $\cech \gamma$ ($-\lambda$ on one side, and either $\lambda$ or $-3\lambda$ on the other side). This concatenation of arcs would then be a thin local set. Proposition \ref {propshift} would then follow from this result in a similar manner as from Lemma \ref {smalllemma}. We leave all this to the interested reader.

\medbreak
As promised in Section \ref {S3}, we now discuss briefly the relation between the boundary-touching level-arcs of $\Gamma$ and those of $\Lambda$. 

Let us give some details for the simplest possible case:
Suppose that one traces a level line at height $0$ of the Dirichlet GFF $\Gamma_j$ inside 
a connected component $O_j$ where the boundary conditions for $\Gamma$ are $-\lambda$. 
This is also clearly locally a $-\lambda$-level line 
of $\Gamma$ as long as it does not hit the ALE. If it hits the ALE, 
then this particular level line of $\Gamma$ stays in $O_j$ and bounces off from the boundary of the ALE; it will in fact create a loop that is entirely contained in $O_j$, and it does so in such a way that the $-2 \lambda$ side of that loop lies in the inside of the loop and not on the side of the ALE (so, for instance, if the level-line is traced in such a way to have $-2 \lambda$ on its left side, then when hitting the ALE, it would have to turn left). This level-line loop of $\Gamma$ is then also exactly the boundary of an ALE-cell of $\Gamma_j$ i.e., 
a level-line of $\Gamma_j$ at height $0$ (in fact, in this particular case, it would actually be a CLE$_4$ loop of $\Gamma$, as mentionned before).
Note already that if one would have followed an arc of the same $0$-level-line of $\Gamma_j$ but if the boundary values of the cell $O_j$ (for $\Gamma$) was $+ \lambda$ instead of $-\lambda$,
then the corresponding level line of $\Gamma$ that this arc is part of would have to turn in the other direction:  
When it hits the ALE, the level line of $\Gamma$ would bounce to the right instead of to the left and it would have traced the boundary of a
neighboring ALE cell of $\Gamma_j$. So, we see that the way in which the
level-arcs of $\Gamma_j$ are hooked-up in order to create level-lines of $\Gamma$ very much depend on the sign of the jumps on the ALE arcs that it hits. 

If we now are looking at the same $0$-level line portion of $\Gamma_j$ and try to see what level line of the Neumann GFF 
$\Lambda$ it will be part of, we see that the interaction 
rule with the ALE will be to bounce in one direction or the other depending on the sign of the height-gap along the ALE arc that it is locally hitting. This shows 
immeditely that the corresponding level line of $\Lambda$ will in fact be forced to end on the boundary of the original domain as soon as it successively hits 
ALE arcs with opposite height-gaps (see Figure \ref {RR} for a sketch). 

Hence, there will exist level arcs with height $\lambda$ for $\Lambda$ that do join two boundary points of the domain. 
As on the other hand, we know that this is not the case for $\Gamma$ (the CLE$_4$ loops do not touch the boundary). We can  conclude that the collection of all level lines ${\mathcal A}_{\lambda}$ with height in $\lambda + 2 \lambda \Z$ for $\Lambda$ is not the same as the collection of all level lines with height in $\lambda + 2 \lambda \Z$ for $\Gamma$ (the former is not empty while the latter is empty). 

These local interaction rules can be easily generalized 
to the interaction between level-lines with height in $a + 2 \lambda \Z$ (and therefore of the ALE ${\mathcal A}_a$) with the level-arcs of the ALE ${\mathcal A}_0$ 
for all $a \in (-2\lambda, 0)$ (the previous discussion was dealing with the case where $a  = - \lambda$).

\section {Further questions and remarks}
We conclude this paper mentioning closely related work and work in progress: 
\begin {enumerate}  
 \item As explained in \cite {Ai,APS}, the nested CLE$_4$ (or nested ALEs) approach to the GFF gives a very natural conformally invariant way to approximate the Liouville measures associated 
 to the GFF. The results of the present paper enable (see also \cite {APS}) to do the same for the boundary Liouville measure for the Neumann GFF. 
 \item It is possible to generalize fairly directly most of the loop-soup part of the arguments to the case of oblique reflection instead of orthogonal reflection. For each given reflection angle $\theta = u \pi$ with $u \in (0 , 1)$, one replaces the orthogonal reflection (corresponding to $\theta = \pi /2$) by reflection with angle $\theta$. In particular, one can then look at the soup  of Brownian loops in the semi-disc $D$, with Dirichlet boundary conditions on $\partial$ and reflected with angle $\theta$ on $I$,  and with intensity $c$. Then, a simple computation shows that:
\begin {lemma}
The upper boundary of the union of all loops in this loop-soup that intersect $I$ 
is a simple curve that satisfies the chordal one-sided conformal restriction property in $D$ with marked point at $-1$ and $1$, with exponent $c u (1-u) / 4$. 
\end {lemma}
One can then in a similar way as in the orthogonal case describe the upper boundary of the union of the loop-soup clusters that touch $I$
as an SLE$_\kappa (\rho)$ process, where $\kappa (c)$ is as before and $\rho$ is the value in $(-2, \infty)$ such that 
$(\rho +2 ) ( \rho + 6 - \kappa )  / ({4 \kappa }) =  c u (1-u) /4$. 
However, caution is needed when one tries to related these loop-soup considerations to Gaussian fields because these reflected Brownian motion are not reversible. 
\item 
Other work in progress related to the present paper involves a loop-soup approach in the spirit of \cite {Wresampling} to SLE-type welding issues. 
Another natural and closely related point under investigation is the relation between such Neumann-Dirichlet type couplings with imaginary flow lines of the GFF (some arguments of the present paper can indeed be fairly directly adapted).  
\end {enumerate}

\subsection*{Acknowledgments}
We acknowledge support of the grant no. 155922 (WQ, WW) and an Early Postdoc.Mobility grant (WQ) of the SNF.
The authors are/were part of NCCR SwissMAP.
We also thank Juhan Aru and Avelio Sep\'ulveda for many stimulating discussions about ALEs, as well as the referees for their comments.


\begin{thebibliography}{99}
\bibitem {Ai}
E. Aidekon.
The extremal process in nested conformal loops. 
Preprint 2015.

\bibitem {ALS}
J. Aru, T. Lupu, and A. Sep\'ulveda. 
 First passage sets of the 2D continuum Gaussian free field.
{\em arXiv e-prints} 2017.

\bibitem {APS}
J. Aru, E. Powell, and A. Sep\'ulveda. 
Approximating Liouville measure using local sets of the Gaussian free field.
{\em arXiv e-prints} 2017.

\bibitem {AS}
J. Aru, A. Sep\'ulveda. 
In preparation. 

\bibitem {ASW}
J. Aru, A. Sep\'ulveda, and W. Werner.
On bounded-type thin local sets (BTLS) of the two-dimensional Gaussian free field.
{\em J. Inst. Math. Jussieu}, to appear.

\bibitem {Be}
V. Beffara. The dimension of the SLE curves,
{\em Ann. Probab.}, 36: 1421-1452, 2008.

\bibitem{BH}
S.~{Benoist} and C.~{Hongler}.
\newblock {The scaling limit of critical Ising interfaces is CLE(3)}.
\newblock {\em Ann. Probab.}, to appear.


\bibitem{cdhks}
D.~Chelkak, H.~Duminil-Copin, C.~Hongler, A.~Kemppainen, and S.~Smirnov.
\newblock Convergence of {I}sing interfaces to {S}chramm's {SLE} curves.
\newblock {\em C. R. Math. Acad. Sci. Paris}, 352(2): 157--161, 2014.

\bibitem{Dub}
J.~Dub\'edat.
\newblock {SLE and the free field: partition functions and couplings}.
\newblock {\em J. American Math. Soc.}, 22: 995-1054, 2009.

\bibitem{MT}
B.~{Duplantier}, J.~{Miller}, and S.~{Sheffield}.
\newblock {Liouville quantum gravity as a mating of trees}.
\newblock {\em arXiv e-prints},  2014.

\bibitem {HK}
C. Hongler and K. Kyt\"ol\"a. 
Ising Interfaces and Free Boundary Conditions.
{\em J. American Math. Soc.} 26: 1107-1189, 2013. 


\bibitem {IK}
K. Izyurov and K. Kyt\"ol\"a. 
Hadamard's formula and couplings of SLEs with free field.
{\em Probab. Th. rel. Fields} 155: 35-69, 2013.

\bibitem {KW}
A. Kemppainen and W. Werner.
The nested conformal loop ensembles in the Riemann sphere. 
{\em Probab. Th. rel. Fields}, {165}: {835--866}, 2016.

\bibitem{LSWr}
G.F.~Lawler, O.~Schramm, and W.~Werner.
\newblock {Conformal restriction: the chordal case}.
\newblock {\em J. American Math. Soc.} 16: 917-955, 2003.

\bibitem{LTF}
G.F. Lawler and J. Trujillo-Ferreras.
Random walk loop soup.
{\em Trans. Amer. Math. Soc.} 359: 767-787, 2007. 

\bibitem{LW}
G.F. Lawler and W.~Werner.
\newblock {The Brownian loop soup}.
\newblock {\em Probab. Th. rel. Fields}, 128: 565--588, 2004.

\bibitem {LJ}
Y. Le Jan.
Markov Paths, Loops and Fields.
L.N. in Math 2026, Springer, 2011.

\bibitem {Lupu}
T. Lupu.
    From loop clusters and random interlacement to the free field, 
{\em Ann. Probab.}, to appear.


\bibitem {Lupu2}
T. Lupu.
    Convergence of the two-dimensional random walk loop soup clusters to CLE,
    {\em J. Europ. Math. Soc.}, to appear. 

    \bibitem {LupuW}
    T. Lupu, W. Werner.
    A note on Ising, random currents, Ising-FK, loop-soups and the GFF, 
   {\em  Electr. Comm. Probab.}, 21, paper no. 13, 2016. 
    
\bibitem {MS}
J.P. Miller and S. Sheffield, 
private communication (2010). 

\bibitem {MS1} J.P. Miller and S. Sheffield.
Imaginary Geometry I. Interacting SLEs, 
{\em Probab. Th. rel. Fields} 64: 553--705, 2016.

\bibitem {MS2} J.P. Miller and S. Sheffield.
Imaginary Geometry II.  Reversibility of SLE$_\kappa (\rho_1;\rho_2)$ for $\kappa \in (0,4)$, 
{\em  Ann. Probab.} 44: 1647-1722, 2016.


\bibitem {MS3} J.P. Miller and S. Sheffield.
Imaginary Geometry III.  Reversibility of SLE$_\kappa$ for $\kappa \in (4,8)$, 
{\em Ann. Math.} 184: 455-486, 2016.


\bibitem {MSW} J.P. Miller, S. Sheffield, and W. Werner. 
CLE percolations, Forum of Mathematics, Pi, Vol. 5, e4, 102 pages (2017). 


\bibitem {MWu}
J.P. Miller and H. Wu. 
Intersections of SLE Paths: the double and cut point
dimension of SLE. 
{\em Probab. Th. rel. Fields} 167: 45--105, 2017.

\bibitem{NW}
{\c{S}}.~Nacu and W.~Werner.
\newblock {Random soups, carpets and fractal dimensions}.
\newblock {\em J. London Math. Soc.},
  83: 789--809, 2011.

\bibitem {PowellWu}
E. Powell, H. Wu. 
Level lines of the GFF with general boundary data. 
{\em Ann. Inst. Henri Poincar\'e}, to appear.


\bibitem{Qian}
W.~{Qian}.
\newblock {Conditioning a Brownian loop-soup cluster on a portion of its
  boundary}.
\newblock
{\em Ann. Inst. Henri Poincar\'e}, to appear.

\bibitem{QianW}
W.~{Qian} and W.~{Werner}.
\newblock {Decomposition of Brownian loop-soup clusters}.
\newblock {\em J. Europ. Math. Soc.}, to appear.

  \bibitem {RS}
  S. Rohde and O. Schramm.
   Basic properties of SLE.
   {\em Ann. Math.}, 161: 883-924, 2005.
   
\bibitem {Sch} O. Schramm.
\newblock {
Scaling limits of loop-erased random walks and uniform spanning trees}.
\newblock {\em Israel J. Math.}, 118: 221--288, 2000.

\bibitem {SchSh}
O. Schramm and S. Sheffield. 
Contour lines of the two-dimensional discrete Gaussian free field.
{\em Acta Math.}, 202: 21--137, 2009.


\bibitem{SchSh2}
O.~Schramm and S.~Sheffield.
     A contour line of the continuous Gaussian free field, 
     {\em Probab. Th. rel. Fields}, 157: 47--80, 2013.
     

\bibitem{SSW}
O.~Schramm, S.~Sheffield, and D.~B. Wilson.
\newblock {Conformal radii for conformal loop ensembles}.
\newblock {\em Comm. Math. Phys.}, 288: 43--53, 2009.


\bibitem{SchrammWilson}
O.~Schramm and D.~B. Wilson.
\newblock S{LE} coordinate changes.
\newblock {\em New York J. Math.}, 11: 659--669, 2005.

\bibitem {Se}
A. Sep\'ulveda.
On thin local sets of the Gaussian Free Field, 
\newblock {\em arXiv e-prints},  2017.

\bibitem{Sh}
S.~Sheffield.
\newblock {Exploration trees and conformal loop ensembles}.
\newblock {\em Duke Math. J.}, 147: 79--129, 2009.


\bibitem{SheffieldQZ}
S.~Sheffield.
\newblock Conformal weldings of random surfaces: {SLE} and the quantum gravity
  zipper.
\newblock {\em Ann. Probab.}, 44(5): 3474--3545, 2016.


\bibitem{ShW}
S.~Sheffield and W.~Werner.
\newblock {Conformal Loop Ensembles: The Markovian characterization and the
  loop-soup construction}.
\newblock {\em Ann. Math.}, 176: 1827--1917, 2012.


\bibitem {WangWu1}
M. Wang and H. Wu. 
 Level Lines of Gaussian Free Field I: Zero-Boundary GFF, {\em Stoch. Proc. Appl.}, to appear.


\bibitem {WangWu}
M. Wang and H. Wu. 
Level Lines of Gaussian Free Field II: Whole-Plane GFF. 
\newblock {\em arXiv e-prints},  2015.

\bibitem{Wcras}
W.~Werner.
\newblock {SLEs as boundaries of clusters of Brownian loops}.
{\em C. R. Math. Acad. Sci. Paris}, 337: 481--486, 2003.

\bibitem {Wcrrq}
W. Werner. Conformal restriction and related questions. 
{\em Probability Surveys}, 2: 145-190, 2005.


\bibitem{Wln}
W.~Werner.
\newblock {Some recent aspects of random conformally invariant systems}.
\newblock Ecole d'\'et\'e de physique des Houches LXXXIII:  57--99, 2006.


\bibitem{Wgff}
W. Werner. 
\newblock {Topics on the Gaussian Free Field and CLE(4)}.
Lecture Notes, 2015.

\bibitem {Wresampling} 
W.~Werner.
On the spatial Markov property of soups of oriented and unoriented loops,
 in {\em S\'em. Probabilit\'es XLVIII}, L.N. in Math. 2168, Springer, 481-503, 2016.
 
 
\bibitem {WW}
W. Werner and H. Wu. 
\newblock {On conformally invariant CLE explorations}.
\newblock {\em Comm. Math. Phys.} 320: 637--661, 2013.

\bibitem {WW2}
W. Werner and H. Wu. 
\newblock {From CLE($\kappa$) to SLE($\kappa, \rho$), }
\newblock {\em Electr. J. Probability} 18, paper 36, 2013.

\end{thebibliography}
\end{document}